\documentclass{amsproc}
\usepackage[all]{xy}

\makeatletter
\newcommand{\xyC}[1]{%
\makeatletter
\xydef@\xymatrixcolsep@{#1}
\makeatother
} % end of \xyR
\makeatletter
\newcommand{\xyR}[1]{%
\makeatletter
\xydef@\xymatrixrowsep@{#1}
\makeatother
} % end of \xyR

\newtheorem{theorem}{Theorem}[section]

\theoremstyle{definition}

\newtheorem{example}[theorem]{Example}

\newtheorem{proposition}{Proposition}

\newtheorem{prop}[proposition]{Proposition}
\newtheorem{thm}[proposition]{Theorem}

\newtheorem{lem}[proposition]{Lemma}
\newtheorem{conj}[proposition]{Conjecture}
\newtheorem{define}[proposition]{Definition}
\newtheorem{rem}[proposition]{Remark}

\theoremstyle{remark}

\numberwithin{equation}{section}

\begin{document}

\title{Eta-invariants and Anomalies in $U(1)$-Chern-Simons theory}

\author{Lisa Jeffrey}
\address{Department of Mathematics, University of Toronto, Toronto,
Ontario, Canada M5S 2E4}
\email{jeffrey@math.toronto.edu}
\thanks{The first author was supported in part by a grant from NSERC}

\author{Brendan McLellan}
\address{Department of Mathematics, University of Toronto, Toronto,
Ontario, Canada M5S 2E4}
\email{mclellan@math.toronto.edu}
\thanks{We would like to thank John Bland, Eckhard Meinrenken,
Rapha\"{e}l Ponge, Edward Witten and especially
Fr\'ed\'eric Rochon and Michel Rumin for
helpful advice related to this work.}

%    General info
\subjclass{Primary 54C40}
\date{January 15, 2010}


\keywords{Contact geometry, quantum field theory}

\begin{abstract}
This paper studies $U(1)$-Chern-Simons theory and its relation to a construction of Chris Beasley and Edward Witten (\cite{bw}).  The natural geometric setup here is that of a three-manifold with a Seifert structure.  Based on a suggestion of Edward Witten we are led to study the stationary phase approximation of the path integral for $U(1)$-Chern-Simons theory after one of the three components of the gauge field is decoupled.  This gives an alternative formulation of the partition function for $U(1)$-Chern-Simons theory that is conjecturally equivalent to the usual $U(1)$-Chern-Simons theory (as in \cite{m}).  The goal of this paper is to establish this conjectural equivalence rigorously through appropriate regularization techniques.   This approach leads to some rather surprising results and opens the door to studying hypoelliptic operators and their associated eta-invariants in a new light.
%% ***   Put your Abstract here.   ***
%% (At most 150 words for M.Sc. or 350 words for Ph.D.)
\end{abstract}

\maketitle

\newtheorem{ex}[proposition]{Example}
\newcommand{\dee}{\text{d}}
\newcommand{\Z}{\mathbb{Z}}
\newcommand{\R}{\mathbb{R}}
\newcommand{\C}{\mathbb{C}}
\newcommand{\I}{\mathbb{I}}
\newcommand{\Q}{\mathbb{Q}}
\newcommand{\U}{\mathbb{U}}
\newcommand{\F}{\mathbb{F}}
\newcommand{\X}{\overline{X}}
\newcommand{\gau}{\mathcal{G}}
\newcommand{\h}{\mathbb{H}}

%\textwidth 5.5 in \oddsidemargin .4 in \topmargin 0 in \headheight
%-.75 in \textheight 10 in

%\begin{document}

\maketitle

%%%%%%%%%%%%%%%%%%%%%  VERSION 2 - UPDATE WHEN FINAL VERSION CHANGED  %%%%%%%%%%%%%%%%%%%%%%%
\newcommand{\ZZ}{{\bf Z}}
%add intro.
%add defn of shift symmetry
\noindent
\section{Introduction}
In \cite{bw} the authors study the Chern-Simons partition function
(see \cite{bw}, (3.1)),
\begin{equation}\label{orgchern}
Z(k)=\frac{1}{\text{Vol}(\mathcal{G})}\left(\frac{k}{4\pi^2}\right)^{\Delta{\mathcal{G}}}\int \mathcal{D}A\,\,\text{exp}\left[i\frac{k}{4\pi}\int_{X}\text{Tr}\left(A\wedge dA+\frac{2}{3}A\wedge A\wedge A\right)\right],
\end{equation}
where,
\begin{itemize}
\item  $A\in\mathcal{A}_{P}=\{ A\in (\Omega^{1}(P)\otimes\frak{g})^{G}\,\,|\,\,A(\xi^{\sharp})=\xi, \,\,\forall\,\xi\in\frak{g}\}$ is a connection on a principal $G$-bundle $\pi:P\rightarrow X$\footnote{In fact, \cite{bw} consider only $G$ compact, connected and simple, and for concreteness one may assume $G=SU(2)$.} over a closed three-manifold $X$,
\item  $\frak{g}=\text{Lie{G}}$ and $\xi^{\sharp}\in\Gamma(TX)$ is the vector field on $P$ generated by the infinitesimal action of $\xi$ on $P$,

\item  $k\in\Z$ (thought of as an element of $H^{4}(BG,\Z)$ that parameterizes the possible Chern-Simons invariants),

\item  $\mathcal{G}:=\{\psi\in(\text{Diff}(P,P))^{G}\,\,|\,\,\pi\circ\psi=\pi\}$ is the \emph{gauge group},

\item  $\Delta(\mathcal{G})$ is formally defined as the dimension of the gauge group.\footnote{Note that the definition of the Chern-Simons partition function in Eq. \ref{orgchern} is completely heuristic.  The measure $\mathcal{D}A$ has not been defined, but only assumed to ``exist heuristically,'' and the volume and dimension of the gauge group, $\text{Vol}(\mathcal{G})$ and $\Delta(\mathcal{G})$, respectively, are at best formally defined.}
\end{itemize}

In general, the partition function of Eq. \ref{orgchern} does not admit a general mathematical interpretation in terms of the cohomology of some classical moduli space of connections, in contrast to Yang-Mills theory for example (cf. \cite{w2}).  The main result of \cite{bw}, however, is that if $X$ is assumed to carry the additional geometric structure of a Seifert manifold, then the partition function of Eq. \ref{orgchern} \emph{does} admit a more conventional interpretation in terms of the cohomology of some classical moduli space of connections.
%More properly, we should note that \cite{bw} consider ``CR-Seifert'' manifolds:
%\begin{define}
%A CR-Seifert manifold is a three-dimensional compact manifold endowed with both a pseudoconvex CR structure $(H,J)$ and a Seifert structure, that are compatible in %the sense that the circle action $\phi:U(1)\rightarrow \text{Diff}(X)$ preserves the CR structure and is generated by a Reeb field $R$.  In particular, given a %choice of contact form $\kappa$, the Reeb field is Killing for the associated metric $g=\kappa\otimes\kappa+d\kappa(\cdot,J\cdot)$.
%\end{define}
Using the additional Seifert structure on $X$, \cite{bw} decouple one of the components of a gauge field $A$, and introduce a new partition function (cf. \cite{bw} ; Eq. 3.7),
\begin{equation}\label{newchern}
\bar{Z}(k)=K\cdot\int \mathcal{D}A\mathcal{D}\Phi\,\,\text{exp}\left[i\frac{k}{4\pi}\left(CS(A)-\int_{X}2\kappa\wedge\text{Tr}(\Phi F_{A})+\int_{X}\kappa\wedge d\kappa\,\,\text{Tr}(\Phi^{2})\right)\right],
\end{equation}
where
\begin{itemize}
\item
$K:=\frac{1}{\text{Vol}(\mathcal{G})}\frac{1}{\text{Vol}(\mathcal{S})}\left(\frac{k}{4\pi^2}\right)^{\Delta{\mathcal{G}}}$,
\item  $\kappa\in\Omega^{1}(X,\R)$ is a contact form associated to the Seifert fibration of $X$ (cf. \cite{bw} ; \S 3.2),
\item  $\Phi\in\Omega^{0}(X,\frak{g})$ is a Lie algebra-valued zero form on $X$,
\item  $\mathcal{D}\Phi$ is a measure on the space of fields $\Phi$,\footnote{The measure $\mathcal{D}\Phi$ is defined independently of any metric on $X$ and is formally defined by the positive definite quadratic form
\begin{equation*}
(\Phi,\Phi):=-\int_{X}\kappa\wedge d\kappa\,\,\text{Tr}(\Phi^{2}),
\end{equation*}
which is invariant under the choice of representative for the contact structure $(X,H)$ on $X$, i.e. under the scaling $\kappa\mapsto f\kappa$, $\Phi\mapsto f^{-1}\Phi$, for some non-zero function $f\in\Omega^{0}(X,\R)$.}
\item  $\mathcal{S}$ is the space of local \emph{shift symmetries}\footnote{$\mathcal{S}$ may be identified with $\Omega^{0}(X,\frak{g})$, where the ``action'' on $\mathcal{A}_P$ is defined as $\delta_{\sigma}(A):=\sigma\kappa$, and on the space of fields $\Phi$ is defined as $\delta_{\sigma}(\Phi):=\sigma$, for $\sigma\in\Omega^{0}(X,\frak{g})$.  $\delta_{\sigma}$ denotes the action associated to $\sigma$.} that ``acts'' on the space of connections $\mathcal{A}_{P}$ and the space of fields $\Phi$ (cf. \cite{bw} ; \S 3.1),
\item $F_{A}\in\Omega^{2}(X,\frak{g})$ is the curvature of $A$, and
\item  $CS(A):=\int_{X}\text{Tr}\left(A\wedge dA+\frac{2}{3}A\wedge A\wedge A\right)$ is the Chern-Simons action.
\footnote{Note that the partition functions of Eq.'s \ref{orgchern} and \ref{newchern} are defined implicitly with respect the pullback of some trivializing section of the principal $G$-bundle $P$.  Of course, every principal $G$-bundle over a three-manifold for $G$ compact, connected and simple is trivializable.  It is basic fact that the partition functions of Eq.'s \ref{orgchern} and \ref{newchern} are independent of the choice of such trivializations.}
\end{itemize}
  \cite{bw} then give a heuristic argument showing that the partition function computed using the alternative description of Eq. \ref{newchern} should be the same as the Chern-Simons partition function of Eq. \ref{orgchern}.  In essence, they show
\begin{equation}\label{maint}
Z(k)=\bar{Z}(k),
\end{equation}
by gauge fixing $\Phi=0$ using the shift symmetry.  \cite{bw} then observe that the $\Phi$ dependence in the integral can be eliminated by simply performing the Gaussian integral over $\Phi$ in Eq. \ref{newchern} directly.  They obtain the alternative formulation:
\begin{equation}\label{newchern2}
Z(k)=\bar{Z}(k)=K'\cdot\int \mathcal{D}A\,\,\text{exp}\left[i\frac{k}{4\pi}\left(CS(A)-\int_{X}\frac{1}{\kappa\wedge d\kappa}\,\,\text{Tr}\left[(\kappa\wedge F_{A})^{2}\right]\right)\right],
\end{equation}
where $K':=\frac{1}{\text{Vol}(\mathcal{G})}\frac{1}{\text{Vol}(\mathcal{S})}\left(\frac{-ik}{4\pi^2}\right)^{\Delta{\mathcal{G}}/2}$.\footnote{Note that we have abused notation slightly by writing $\frac{1}{\kappa\wedge d\kappa}$.  We have done this with the understanding that since $\kappa\wedge d\kappa$ is non-vanishing (since $\kappa$ is a contact form), then $\kappa\wedge F_{A}=\phi\,\kappa\wedge d\kappa$ for some function $\phi\in \Omega^{0}(X,\frak{g})$, and we identify $\frac{\kappa\wedge F_{A}}{\kappa\wedge d\kappa}:=\phi$.}\\
\\
The objective in this article is to study the partition function for \emph{$U(1)$-Chern-Simons theory} using the analogue of Eq. \ref{newchern2} in this case.  Thus, we are also assuming here that $X$ is a Seifert manifold with a ``compatible'' contact structure, $(X,\kappa)$ (cf. \cite{bw} ; \S 3.2).  Note that any compact, oriented three-manifold possesses a contact structure and one aim of future work is to extend our results to \emph{all} closed three-manifolds using this fact.  For now, we restrict ourselves to the case of closed three-manifolds that possess contact compatible Seifert structures (see Definition \ref{geodef} for example).   We restrict to the gauge group $U(1)$ so that the action is quadratic and hence the stationary phase approximation
is exact.  A salient point is that the group $U(1)$ is not simple, and therefore may have non-trivial principal bundles associated with it.  This makes the $U(1)$-theory very different from the $SU(2)$-theory in that one must now incorporate a sum over bundle classes in a definition of the $U(1)$-partition function.  As an analogue of Eq. \ref{orgchern}, our basic definition of the partition function for $U(1)$-Chern-Simons theory is now
\begin{equation}\label{abelchern1}
Z_{U(1)}(X,k)=\sum_{p\in\text{Tors}H^{2}(X;\Z)}Z_{U(1)}(X,p,k)
\end{equation}
where
\begin{equation}\label{abelchern2}
Z_{U(1)}(X,p,k)=\frac{1}{Vol(\mathcal{G}_{P})}\int_{\mathcal{A}_{P}}\mathcal{D}A e^{\pi i k S_{X,P}(A)},
\end{equation}
recalling that the torsion subgroup
$\text{Tors}H^{2}(X;\Z)< H^{2}(X;\Z)$\footnote{Recall the definition of the
torsion of an abelian group is the collection of those elements which have
finite order.}
enumerates the $U(1)$-bundle classes that have flat connections.  Note that the bundle $P\rightarrow X$ in Eq. \ref{abelchern2} is taken to be any representative of a bundle class with first Chern class $c_{1}(P)=p\in \text{Tors}H^{2}(X;\Z)$.  Also note that some care must be taken to define the Chern-Simons action, $S_{X,P}(A)$, in the case that $G=U(1)$.  We outline this construction in Appendix \ref{appen1}.\\
\\
The main results of this article may be summarized as follows.  First, our main objective is the rigorous confirmation of the heuristic result of Eq. \ref{maint} in the case where the gauge group is $U(1)$.  This statement is certainly non-trivial and involves some fairly deep facts about the ``contact operator'' as studied by Michel Rumin (cf. \cite{r}).  Recall that this is the second order operator ``$D$'' that fits into the complex,
\begin{equation}\label{complex}
C^{\infty}(X)\xrightarrow{\text{$d_{H}$}}\Omega^{1}(H)\xrightarrow{\text{$D$}}\Omega^{2}(V)\xrightarrow{\text{$d_{H}$}}\Omega^{3}(X),
\end{equation}
and is defined by:
\begin{equation}
D\alpha=\kappa\wedge [\mathcal{L}_{\xi}+d_{H}\star_{H} d_{H}]\alpha,\,\,\alpha\in\Omega^{1}(H).
\end{equation}
This operator is elaborated upon in \S\ref{Dsec} below.  A somewhat surprising observation is that this operator shows up quite naturally in $U(1)$-Chern-Simons theory (see Prop. \ref{prop1} below), and this leads us to make several conjectures motivated by the rigorous confirmation of the heuristic result of Eq. \ref{maint}.  Our main result is the following:
\begin{prop}\label{mprop}
Let $(X,\phi,\xi,\kappa,g)$ be a closed, \emph{quasi-regular K-contact} three manifold.  If,
\begin{equation}\label{nz}
\bar{Z}_{U(1)}(X,p,k)=k^{n_X}e^{\pi i k S_{X,P}(A_{0})}e^{\frac{\pi i}{4}\left(\eta(-\star D)+\frac{1}{512}\int_{X}R^{2}\,\,\kappa\wedge d\kappa\right)}\int_{\mathcal{M}_{P}}\,\, (T^{d}_{C})^{1/2}
\end{equation}
where $R\in C^{\infty}(X)$ $=$ the Tanaka-Webster scalar curvature of $X$, and (\cite{m}),
\begin{equation}\label{oz}
Z_{U(1)}(X,p,k)=k^{m_X}e^{\pi i k S_{X,P}(A_{0})}e^{\pi i\left(\frac{\eta(-\star d)}{4}+\frac{1}{12}\frac{\text{CS}(A^{g})}{2\pi}\right)}\int_{\mathcal{M}_{P}}\,\, (T^{d}_{RS})^{1/2}
\end{equation}
then,
\begin{equation*}
Z_{U(1)}(X,k)=\bar{Z}_{U(1)}(X,k)
\end{equation*}
as topological invariants.
\end{prop}

Following \cite{m}, we rigorously define $\bar{Z}_{U(1)}(X,k)$ in \S\ref{partsec} using the fact that the stationary phase approximation for our path integral should be exact.  This necessitates the introduction of the regularized determinant of $D$ in Eq. \ref{regdet}, which in turn naturally involves the hypoelliptic Laplacian of Eq. \ref{maxLap}.  The rigorous quantity that we obtain for the integrand of Eq. \ref{intzeta} in \S\ref{partsec} is derived in Prop. \ref{rigdet}.  Using an observation from \S\ref{gsec} that identifies the volume of the isotropy subgroup of the gauge group $\mathcal{G}_{P}$, we identify the integrand of Eq. \ref{intzeta} with the contact analytic torsion $T^{d}_{C}$ defined in Def. \ref{torsdef}.  After formally identifying the signature of the contact operator $D$ with the $\eta$-invariant of $D$ in \S\ref{esec}, we obtain our fully rigorous definition of $\bar{Z}_{U(1)}(X,k)$ in Eq. \ref{newpar} below, which is repeated in Eq. \ref{nz} above.\\
\\
On the other hand, \cite{m} provides a rigorous definition of the partition function $Z_{U(1)}(X,k)$ that does not involve an \emph{a priori} choice of a contact structure on $X$.  The formula for this is recalled in Eq. \ref{oldpar} below, and is the term $Z_{U(1)}(X,p,k)$ in Eq. \ref{oz} of Prop. \ref{mprop} above.\\
\\
Our first main step in the proof of Prop. \ref{mprop} is confirmation of the fact that the Ray-Singer analytic torsion (cf. \cite{rsi}) of $X$, $T_{RS}^{d}$, is identically equal to the contact analytic torsion $T^{d}_{C}$.\footnote{We consider the square roots thereof, viewed as densities on the moduli space of flat connections $\mathcal{M}_{X}$.}  We observe that this result follows directly from (\cite{rs} ; Theorem 4.2).\\
\\
We also observe in Remark \ref{rmknX} that the quantities $m_{X}$ and $n_{X}$ that occur in Prop. \ref{mprop} are also equal.  This leaves us with the main final step in the confirmation of Prop. \ref{mprop}, which involves a study of the $\eta$-invariants, $\eta(-\star d)$, $\eta(-\star D)$, that naturally show up in $Z_{U(1)}(X,k)$, $\bar{Z}_{U(1)}(X,k)$, respectively.  This analysis is carried out in \S\ref{fsec}, where we observe that the work of Biquard, Herzlich, and Rumin (\cite{bhr}) is our most pertinent reference.  Our main observation here is that the quantum anomalies that occur in the computation of $Z_{U(1)}(X,k)$ and $\bar{Z}_{U(1)}(X,k)$ should, in an appropriate sense, be completely equivalent.  In our case, these quantum anomalies are made manifest precisely in the failure of the $\eta$-invariants to represent topological invariants.  As observed by Witten (cf. \cite{w3}), this is deeply connected with the fact that in order to \emph{actually} compute the partition function, one needs to make a choice that is tantamount to either a valid gauge choice for representatives of gauge classes of connections, or in some other way by breaking the symmetry of our problem.  Such a choice for us is equivalent to a choice of metric, which is encoded in the choice of a quasi-regular K-contact structure on our manifold $X$.  Witten observes in \cite{w3} that the quantum anomaly that is introduced by our choice of metric may be canceled precisely by adding an appropriate ``counterterm'' to the $\eta$-invariant, $\eta(-\star d)$.  This recovers topological invariance and effectively cancels the anomaly.\footnote{In this case, topological invariance is recovered only up to a choice of two-framing for $X$.  Of course, there is a canonical choice of such framing (\cite{at}), and we assume this choice throughout.}    This counterterm is found by appealing to the Atiyah-Patodi-Singer theorem, and is in fact identified as the gravitational Chern-Simons term
\begin{equation}
\text{CS}(A^{g}):=\frac{1}{4\pi}\int_{X}Tr(A^{g}\wedge dA^{g}+\frac{2}{3} A^{g}\wedge A^{g}\wedge A^{g}),
\end{equation}
where $A^{g}$ is the Levi-Civita connection on the spin bundle of $X$ for the metric,
\begin{equation}
g=\kappa\otimes\kappa+d\kappa(\cdot,J\cdot),
\end{equation}
on our quasi-regular K-contact three manifold, $(X,\phi,\xi,\kappa,g)$.  In particular, we use the fact that,
\begin{equation}\label{regg}
\frac{\eta(-\star d)}{4}+\frac{1}{12}\frac{\text{CS}(A^{g})}{2\pi},
\end{equation}
is a topological invariant of $X$, after choosing the canonical framing.
As is discussed in \S\ref{fsec}, this leads us to conjecture that there exists an appropriate counterterm for the $\eta$-invariant associated to the contact operator $D$ that yields the same topological invariant as in Eq. \ref{regg}.  More precisely, we conjecture that there exists a counterterm, $C_{T}$, such that
\begin{equation}
e^{\pi i\left[\frac{\eta(-\star d)}{4}+\frac{1}{12}\frac{\text{CS}(A^{g})}{2\pi}\right]}=e^{\frac{\pi i}{4}\left[\eta(-\star_{H} D^{1})+C_{T}\right]},
\end{equation}
as topological invariants.  We establish the following in Proposition \ref{lprop},
\begin{prop}
$(X,\phi,\xi,\kappa,g)$ closed, quasi-regular K-contact three-manifold.  The counterterm, $C_{T}$, such that $e^{\frac{\pi i}{4}\left[\eta(-\star_{H} D^{1})+C_{T}\right]}$ is a topological invariant that is identically equal to the topological invariant $e^{\pi i\left[\frac{\eta(-\star d)}{4}+\frac{1}{12}\frac{\text{CS}(A^{g})}{2\pi}\right]}$ is $$C_{T}=\frac{1}{512}\int_{X}R^{2}\,\,\kappa\wedge d\kappa,$$
where $R\in C^{\infty}(X)$ is the Tanaka-Webster scalar curvature of $X$.
\end{prop}
This proposition is proven in \S\ref{fsec} by appealing to the following result, which is established using a ``Kaluza-Klein'' dimensional reduction technique for the gravitational Chern-Simons term.  This result is modeled after the paper \cite{gijp}, and is listed as Proposition \ref{mc2}.
\begin{prop}(\cite{mcl2})
$(X,\phi,\xi,\kappa,g)$ closed, quasi-regular K-contact three-manifold,
\begin{displaymath}
\xymatrix{\xyC{2pc}\xyR{1pc}U(1) \ar@{^{(}->}[r] & X \ar[d]\\
                              & \Sigma}.
\end{displaymath}
Let $g_{\epsilon}:=\epsilon^{-1}\,\kappa\otimes\kappa+\pi^{*}h$.  After choosing a framing for $TX\oplus TX$, corresponding to a choice of vielbeins, then,
\begin{equation}
CS(A^{g_{\epsilon}})=\left(\frac{\epsilon^{-1}}{2}\right)\int_{\Sigma}r\,\omega+\left(\frac{\epsilon^{-2}}{2}\right)\int_{\Sigma}f^{2}\,\omega
\end{equation}
where $r\in C^{\infty}_{orb}(\Sigma)$ is the (orbifold) scalar curvature of $(\Sigma,h)$, $\omega\in\Omega^{2}_{orb}(\Sigma)$ is the (orbifold) Hodge form of $(\Sigma,h)$, and $f:=\star_{h}\omega$.  In particular, the adiabatic limit of $\text{CS}(A^{g_{\epsilon}})$ vanishes:
\begin{equation}
\lim_{\epsilon\rightarrow \infty}\text{CS}(A^{g_{\epsilon}})=0.
\end{equation}
\end{prop}
Finally, as a consequence of these investigations, we are able to compute in Proposition \ref{cprop} the $U(1)$-Chern-Simons partition function fairly explicitly.
\begin{prop}
$(X,\phi,\xi,\kappa,g)$ closed, quasi-regular K-contact three-manifold.  Then,
\begin{eqnarray*}
\eta(-\star d)+\frac{1}{3}\frac{\text{CS}(A^{g})}{2\pi}&=&\eta(-\star D)+\frac{1}{512}\int_{X}R^{2}\,\,\kappa\wedge d\kappa\\
                                                       &=&1+\frac{d}{3}+4\sum_{j=1}^{N}s(\alpha_{j},\beta_{j}),
\end{eqnarray*}
where $d=c_1(X)=n+\sum_{j=1}^{N}\frac{\beta_{j}}{\alpha_j}\in\Q$ and
\begin{equation*}
s(\alpha,\beta):=\frac{1}{4\alpha}\sum_{k=1}^{\alpha-1}cot\left(\frac{\pi k}{\alpha}\right)cot\left(\frac{\pi k\beta}{\alpha}\right)\in\Q
\end{equation*}
is the classical Rademacher-Dedekind sum, where $[n; (\alpha_{1},\beta_{1}),\ldots,(\alpha_{N},\beta_{N})]$ (for gcd$(\alpha_{j},\beta_{j})=1$) are the Seifert invariants of $X$.  In particular, we have computed the $U(1)$-Chern-Simons partition function as:
\begin{eqnarray*}
Z_{U(1)}(X,p,k)&=&k^{n_X}e^{\pi i k S_{X,P}(A_{0})}e^{\frac{\pi i}{4}\left(1+\frac{d}{3}+4\sum_{j=1}^{N}s(\alpha_{j},\beta_{j})\right)}\int_{\mathcal{M}_{P}}\,\, (T^{d}_{C})^{1/2},\\
               &=&k^{m_X}e^{\pi i k S_{X,P}(A_{0})}e^{\frac{\pi i}{4}\left(1+\frac{d}{3}+4\sum_{j=1}^{N}s(\alpha_{j},\beta_{j})\right)}\int_{\mathcal{M}_{P}}\,\, (T^{d}_{RS})^{1/2}.
\end{eqnarray*}

\end{prop}

\section{Preliminary Results}
Our starting point is the analogue of Eq. \ref{newchern2} for the $U(1)$-Chern-Simons partition function:
\begin{equation}\label{Anom1}
\bar{Z}_{U(1)}(X,p,k)=\frac{e^{\pi i k S_{X,P}(A_{0})}}{Vol(\mathcal{S})Vol(\mathcal{G}_{P})}\int_{\mathcal{A}_{P}}DA\,\, exp\,\left[\frac{i k}{4\pi}\left(\int_{X} A\wedge dA-\int_{X} \frac{(\kappa\wedge dA)^{2}}{\kappa\wedge d\kappa}\right)\right]
\end{equation}
%where $\mathcal{S}$ is the space of shift symmetries and its volume is formally defined by the quadratic form (\cite{bw} ; Eq. 3.8)
%\begin{equation}\label{Anom2}
%(\Phi,\Phi)=-\int_{X}\kappa\wedge d\kappa\,\, Tr(\Phi^{2}).
%\end{equation}
where $S_{X,P}(A_{0})$ is the Chern-Simons invariant associated to $P$ for $A_{0}$ a flat connection on $P$.
The derivation of Eq. \ref{Anom1} can be found in Appendix \ref{appen1}.  It is obtained by expanding the $U(1)$ analogue of Eq. \ref{newchern2} around a critical point $A_{0}$ of the action.  Note that the critical points of this action, up to the action of the shift symmetry, are precisely the flat connections (\cite{bw} ; Eq. 5.3).  In our notation, $A\in T_{A_{0}}\mathcal{A}_{P}$.  Let us define the notation
\begin{equation}\label{newact}
S(A):=\int_{X} A\wedge dA-\int_{X} \frac{(\kappa\wedge dA)^{2}}{\kappa\wedge d\kappa}
\end{equation}
for the new action that appears in the partition function.  Also, define
\begin{equation}
\bar{S}(A):=\int_{X} \frac{(\kappa\wedge dA)^{2}}{\kappa\wedge d\kappa}
\end{equation}
so that we may write
\begin{equation}
S(A)=CS(A)-\bar{S}(A)
\end{equation}
The primary virtue of Eq. \ref{Anom1} above is that it is exactly equal to the original Chern-Simons partition function of Eq. \ref{abelchern2} and yet it is expressed in such a way that the action $S(A)$ is invariant under the shift symmetry.  This means that $S(A+\sigma\kappa)=S(A)$ for all tangent vectors $A\in T_{A_{0}}(\mathcal{A}_{P})\simeq\Omega^{1}(X)$ and $\sigma\in \Omega^{0}(X)$.  We may naturally view $A\in\Omega^1(H)$, the sub-bundle of $\Omega^{1}(X)$ restricted to the contact distribution $H\subset TX$.  Equivalently, if $\xi$ denotes the Reeb vector field of $\kappa$, then $\Omega^{1}(H)=\{\omega\in\Omega^{1}(X)\,\,|\,\,\iota_{\xi}\omega=0\}$.  The remaining contributions to the partition function come from the orbits of $\mathcal{S}$ in $\mathcal{A}_{P}$, which turn out to give a contributing factor of $Vol(\mathcal{S})$ (cf. \cite{bw} ; Eq. 3.32).  We thus reduce our integral to an integral over $\bar{\mathcal{A}}_{P}:=\mathcal{A}_{P}/\mathcal{S}$ and obtain:
\begin{eqnarray*}
Z_{U(1)}(X,p,k)&=&\frac{e^{\pi i k S_{X,P}(A_{0})}}{Vol(\mathcal{G}_{P})}\int_{\bar{\mathcal{A}}_{P}}\bar{D}A\,\, exp\,\left[\frac{i k}{4\pi}\left(\int_{X} A\wedge dA-\int_{X} \frac{(\kappa\wedge dA)^{2}}{\kappa\wedge d\kappa}\right)\right]\\
               &=&\frac{e^{\pi i k S_{X,P}(A_{0})}}{Vol(\mathcal{G}_{P})}\int_{\bar{\mathcal{A}}_{P}}\bar{D}A\,\, exp\,\left[\frac{i k}{4\pi}S(A)\right]
\end{eqnarray*}\\
where $\bar{D}A$ denotes an appropriate quotient measure on $\bar{\mathcal{A}}_{P}$, and we can now assume that $A\in\Omega^1(H)\simeq T_{A_{0}}\bar{\mathcal{A}}_{P}$.\\
\\
\noindent
\section{Contact structures}
At this point, we further restrict the structure on our $3$-manifold and assume that the Seifert structure is compatible with a contact metric structure $(\phi,\xi,\kappa,g)$ on $X$.  In particular, we restrict to the case of a quasi-regular K-contact manifold.  Let us review some standard facts about these structures in the case of dimension three.
\begin{rem}
Our three manifolds $X$ are assumed to be closed throughout this paper.
\end{rem}
\begin{define}\label{eq1}
A \emph{K-contact} manifold is a manifold $X$ with a contact metric structure $(\phi,\xi,\kappa,g)$ such that the Reeb field $\xi$ is Killing for the associated metric $g$, $\mathcal{L}_{\xi}g=0$.
\end{define}
\noindent
where,
\begin{itemize}
\item  $\kappa\in\Omega^{1}(X)$ contact form, $\xi=$ Reeb vector field.

\item  $H:=\text{ker}\kappa\subset TX$ denotes the horizontal or contact distribution on $(X,\kappa)$.

\item  $\phi\in \text{End}(TX)$, $\phi(Y)=JY$ for $Y\in \Gamma(H)$, $\phi(\xi)=0$ where $J\in \text{End}(H)$ complex structure on the contact distribution $H\subset TX$.

\item  $g=\kappa\otimes\kappa+d\kappa(\cdot, \phi\cdot)$

\end{itemize}
\noindent
\begin{rem}
Note that we will assume that our contact structure is ``co-oriented,'' meaning that the contact form $\kappa\in\Omega^{1}(X)$ is a global form.  Generally, one can take the contact structure to be to be defined only locally by the condition $H:=\text{ker}\,\kappa$, where $\kappa\in\Omega^{1}(U)$ for open subsets $U\in X$ contained in an open cover of $X$.
\end{rem}
\begin{define}\label{eq2}
The characteristic foliation $\mathcal{F}_{\xi}$ of a contact manifold $(X,\kappa)$ is said to be \emph{quasi-regular} if there is a positive integer $j$ such that each point has a foliated coordinate chart $(U,x)$ such that each leaf of $\mathcal{F}_{\xi}$ passes through $U$ at most $j$ times.  If $j=1$ then the foliation is said to be \emph{regular}.
\end{define}
\noindent
Definitions \ref{eq1} and \ref{eq2} together define a quasi-regular $K$-contact manifold, $(X,\phi,\xi,\kappa,g)$.  Such three-manifolds are necessarily ``Seifert'' manifolds that fiber over a two dimensional orbifold $\widehat{\Sigma}$ with with some additional structure.  Recall:
\begin{define}
A \emph{Seifert manifold} is a three manifold $X$ that admits a locally free $U(1)$-action.
\end{define}
\noindent
Thus, Seifert manifolds are simply $U(1)$-bundles over an orbifold $\widehat{\Sigma}$,
\begin{displaymath}
\xymatrix{\xyC{2pc}\xyR{1pc}U(1) \ar@{^{(}->}[r] & X \ar[d]\\
                              & \widehat{\Sigma}}.
\end{displaymath}
\noindent
We have the following classification result:  $X$ is a quasi-regular K-contact three manifold $\iff$
\begin{itemize}
\item  (\cite{bg}; Theorem 7.5.1, (i)) $X$ is a $U(1)$-Seifert manifold over a Hodge orbifold surface, $\widehat{\Sigma}$.

\item  (\cite{bg}; Theorem 7.5.1, (iii)) $X$ is a $U(1)$-Seifert manifold over a normal projective algebraic variety of real dimension two.
\end{itemize}
\begin{example}
All 3-dimensional Lens spaces, $L(p,q)$ and the Hopf fibration $S^{1}\hookrightarrow S^{3}\rightarrow \C\mathbb{P}^{1}$ possess quasi-regular K-contact structures.  Note that any trivial $U(1)$-bundle over a Riemann surface $\Sigma_{g}$, $X=U(1)\times \Sigma_{g}$, possesses \emph{no} K-contact structure (\cite{itoh}), however, and our results do not apply in this case.
\end{example}
\begin{rem}
Note that in fact our results apply to the class of all closed \emph{Sasakian} three-manifolds.  This follows from the observation that every Sasakian three manifold is K-contact (cf. \cite{b} ; Corollary 6.5), and every K-contact manifold possesses a quasi-regular K-contact structure (cf. \cite{bg} ; Theorem 7.1.10).
\end{rem}
\noindent
A useful observation for us is that for a quasi-regular K-contact three-manifold, the metric tensor $g$ must take the following form (cf. \cite{bg} ; Theorem 6.3.6):
\begin{equation}
g=\kappa\otimes\kappa+\pi^{*}h
\end{equation}
where $\pi:X\rightarrow \Sigma$ is our quotient map, and $h$ represents any (orbifold)K\"{a}hler metric on $\widehat{\Sigma}$ which is normalized so that the corresponding (orbifold)K\"{a}hler form, $\widehat{\omega}\in\Omega^{2}_{orb}(\Sigma,\R)$, pulls back to $d\kappa$.\\
\noindent
Note that the assumption that the Seifert structure on $X$ comes from a quasi-regular K-contact structure $(\phi,\xi,\kappa,g)$ is equivalent to assuming that $X$ is a $CR$-Seifert manifold (cf. \cite{bg} ; Prop. 6.4.8).  Recall the following
\begin{define}\label{geodef}
A \emph{CR-Seifert} manifold is a three-dimensional compact manifold endowed with both a strictly pseudoconvex CR structure $(H,J)$ and a Seifert structure, that are compatible in the sense that the circle action $\psi:U(1)\rightarrow \text{Diff}(X)$ preserves the CR structure and is generated by a Reeb field $\xi$.  In particular, given a choice of contact form $\kappa$, the Reeb field is Killing for the associated metric $g=\kappa\otimes\kappa+d\kappa(\cdot,J\cdot)$.
\end{define}
The assumption that $X$ is CR-Seifert (hence quasi-regular K-contact) is sufficient to ensure that the assumption in (\cite{bw} ; Eq. 3.27), which states that the $U(1)$-action on $X$, $\psi:U(1)\rightarrow \text{Diff}(X)$, acts by isometries, is satisfied.\\
\\
We now employ the natural Hodge star operator $\star$, induced by the metric $g$ on $X$, that acts on $\Omega^{\bullet}(X)$ taking $k$ forms to $3-k$ forms.  As a result of this normalization convention, we have $\star 1=\kappa\wedge d\kappa$ and $\star\kappa=d\kappa$.  Now let
\begin{equation}
\star_{H}=-\iota_{\xi}\circ\star
\end{equation}
as in equation (3.30) of \cite{bw}.  This operator then satisfies
\begin{eqnarray}
\star_{H}\kappa&=&0\\
\star_{H}(\kappa\wedge d\kappa)&=&0\\
\star_{H} 1&=&-d\kappa\\
(\star_{H})^{2}&=&-1
\end{eqnarray}
as is shown in (\cite{bw} ; pg. 20).  We also define a horizontal exterior derivative $d_H$ as the usual exterior derivative $d$ restricted to the space of horizontal forms $\Omega^{\bullet}(H)$.\\
\\
\noindent
Our key observation is that the action $S(A)$ may now be expressed in terms of these horizontal quantities.  Let us start with the term $\bar{S}(A)$.  Firstly, the term $\kappa\wedge dA$ in $\bar{S}(A)$ is equivalent to $\kappa\wedge d_{H}A$ since the vertical part of $dA$ is annihilated by $\kappa$ in the wedge product.  The term $\frac{\kappa\wedge dA}{\kappa\wedge d\kappa}$ is equivalent to $\star(\kappa\wedge d_{H}A)$ by the properties of $\star$ above.  By the definition of $\star_{H}$, $\star(\kappa\wedge d_{H}A)=\star_{H}d_{H}A$.  We then have,
\begin{eqnarray*}
\bar{S}(A)&=&\int_{X} \frac{(\kappa\wedge dA)^{2}}{\kappa\wedge d\kappa}\\
          &=&\int_{X} \star_{H}(d_{H}A)\wedge \kappa \wedge d_{H}A\\
          &=&\int_{X} \kappa \wedge [d_{H}A \wedge\star_{H}(d_{H}A)]\\
\end{eqnarray*}
We claim that $\bar{S}(A)$ is now expressed in terms of an inner product on $\Omega^{2}{H}$.  More generally, we define an inner product on $\Omega^{l}(H)$ for $0\leq l\leq 2$:
\begin{define}
Define the pairing $\langle\cdot,\cdot\rangle^{l}_{\kappa}:\Omega^{l}{H}\times\Omega^{l}{H}\rightarrow\R$ as
\begin{equation}\label{inner}
\langle\alpha,\beta\rangle^{l}_{\kappa}:=(-1)^{l}\int_{X} \kappa \wedge [\alpha\wedge\star_{H}\beta]
\end{equation}
for any $\alpha, \beta\in \Omega^{l}{H}$, $0\leq l\leq 2$.
\end{define}
\begin{prop}
The pairing $\langle\cdot,\cdot\rangle^{l}_{\kappa}$ is an inner product on $\Omega^{l}{H}$.
\end{prop}
\begin{proof}
It can be easily checked that this pairing is just the restriction of the usual $L^{2}$-inner product, $\langle \cdot, \cdot \rangle:\Omega^{l}{X}\times\Omega^{l}{X}\rightarrow\R$,
\begin{equation}
\langle\alpha,\beta\rangle:=\int_{X} \alpha\wedge\star\beta
\end{equation}
restricted to horizontal forms.  i.e. for any $\beta\in \Omega^{l}{H}$, $0\leq l\leq 2$, we have $\star\beta=\kappa\wedge\star_{H}\beta$.  We then have $\alpha\wedge\star\beta=(-1)^{l}\kappa \wedge [\alpha\wedge\star_{H}\beta]$ for any $\alpha, \beta\in \Omega^{l}{H}$, $0\leq l\leq 2$.  Thus, $\langle\cdot,\cdot\rangle^{l}_{\kappa}=\langle \cdot, \cdot \rangle$ on $\Omega^{l}{H}$ and therefore defines an inner product.
\end{proof}
\noindent
By our definition, we may now write $\bar{S}(A)=\langle d_{H}A, d_{H}A\rangle^{2}_{\kappa}$.  We make the following
\begin{define}
Define the formal adjoint of $d_{H}$, denoted $d_{H}^{*}$, via:
\begin{equation*}
\langle d_{H}^{*}\gamma,\phi\rangle^{l-1}_{\kappa}=\langle \gamma,d_{H}\phi\rangle^{l}_{\kappa}
\end{equation*}
for $\gamma\in\Omega^{l}(H)$, $\phi\in\Omega^{l-1}(H)$ where $l=1,2$ and $d_{H}^{*}\gamma=0$ for $\gamma\in\Omega^{0}(H)$.
\end{define}
\begin{prop}
$d_{H}^{*}=(-1)^{l}\star_{H}d_{H}\star_{H}:\Omega^{l}(H)\rightarrow\Omega^{l-1}(H)$, $0\leq l \leq 2$, where $\Omega^{-1}(H):=0$.
\end{prop}
\begin{proof}
This just follows from the definition of $d^{*}$ relative to the ordinary inner product $\langle \cdot,\cdot\rangle$, and the facts that $\langle \cdot, \cdot\rangle^{l-1}_{\kappa}$ is just this ordinary inner product restricted to horizontal forms and $d^{*}=(-1)^{l}\star d\star$.
\end{proof}
\noindent
Thus, we may now write $\bar{S}(A)=\langle A, d_{H}^{*}d_{H}A\rangle^{1}_{\kappa}$ and identify this piece of the action with the second order operator $d_{H}^{*}d_{H}$ on horizontal forms.\\
\\
\noindent
Now we turn our attention to the Chern-Simons part of the action $CS(A)=\int_{X} A\wedge dA$.  We would like to reformulate this in terms of horizontal quantities as well.  This is straightforward to do;  simply observe that $dA=\kappa\wedge\mathcal{L}_{\xi}A+d_{H}A$.  Thus, we have:
\begin{eqnarray}
CS(A)&=&\int_{X} A\wedge dA\\
     &=&\int_{X} A\wedge [\kappa\wedge\mathcal{L}_{\xi}A+d_{H}A]\\
     &=&\int_{X} A\wedge [\kappa\wedge\mathcal{L}_{\xi}A]+\int_{X} A\wedge d_{H}A\\
     &=&\int_{X} A\wedge [\kappa\wedge\mathcal{L}_{\xi}A]
\end{eqnarray}
where the last line follows from the fact that $A\wedge d_{H}A=0$ since both forms are horizontal.  Putting this all together,
we may now express the total action $S(A)$ in terms of horizontal quantities as follows:
\begin{eqnarray*}
S(A)&=&CS(A)-\bar{S}(A)\\
    &=&\int_{X} A\wedge [\kappa\wedge\mathcal{L}_{\xi}A]+\int_{X} A\wedge [\kappa\wedge d_{H}\star_{H} d_{H} A]\\
    &=&\int_{X} A\wedge [\kappa\wedge (\mathcal{L}_{\xi}+d_{H}\star_{H} d_{H})A]
\end{eqnarray*}
\section{The contact operator $D$}\label{Dsec}

A surprising observation is that $\kappa\wedge (\mathcal{L}_{\xi}+d_{H}\star_{H} d_{H})$ turns out to be well known.  It is the second order operator ``$D$'' that fits into the complex,
\begin{equation}\label{complex}
C^{\infty}(X)\xrightarrow{\text{$d_{H}$}}\Omega^{1}(H)\xrightarrow{\text{$D$}}\Omega^{2}(V)\xrightarrow{\text{$d_{H}$}}\Omega^{3}(X)
\end{equation}
where,
\begin{equation}
\Omega^{\bullet}(V):=\{\kappa\wedge\alpha\,\,|\,\,\alpha\in\Omega^{\bullet}(H)\}=\kappa\wedge\Omega^{\bullet}(H)
\end{equation}
and for $f\in C^{\infty}(X)$, $d_{H}f\in\Omega^{1}(H)$ stands for the restriction of $df$ to $H$ as usual, while
\begin{equation}
d_{H}:\Omega^{2}(V)\rightarrow \Omega^{3}(X)
\end{equation}
is just de Rham's differential restricted to $\Omega^{2}(V)$ in $\Omega^{2}(X)$.  $D$ is defined as follows:  since $d$ induces an isomorphism
\begin{equation}
d_{0}:\Omega^{1}(V)\rightarrow\Omega^{2}(H),\,\,\text{with}\,\,d_{0}(f\kappa)=fd\kappa|_{\Lambda^{2}(H)}
\end{equation}
then any $\alpha\in\Omega^{1}(H)$ admits a unique extension $\textit{l}(\alpha)$ in $\Omega^{1}(X)$ such that $d\textit{l}(\alpha)$ belongs to $\Omega^{2}(V)$;  i.e. given any initial extension $\bar{\alpha}$ of $\alpha$, one has
\begin{equation}
\textit{l}(\alpha)=\bar{\alpha}-d_{0}^{-1}(d\bar{\alpha})|_{\Lambda^{2}(H)}
\end{equation}
We then define
\begin{equation}
D\alpha:=d\textit{l}(\alpha)
\end{equation}
We then have (\cite{bhr} ; Eq. 39),
\begin{equation}\label{Ddef1}
D\alpha=\kappa\wedge [\mathcal{L}_{\xi}+d_{H}\star_{H} d_{H}]\alpha
\end{equation}
for any $\alpha\in\Omega^{1}(H)$.
Thus,
\begin{eqnarray}
S(A)&=&\int_{X} A\wedge [\kappa\wedge (\mathcal{L}_{\xi}+d_{H}\star_{H} d_{H})A]\\
    &=&\int_{X} A\wedge DA\\
    &=&\langle A, -\star DA\rangle
\end{eqnarray}
where $\langle \cdot, \cdot\rangle$ is the usual $L^{2}$ inner product on $\Omega^{1}(X)$.\\
\\
\noindent
Alternatively, we make the following
\begin{define}
Let $D^{1}:\Omega^{1}(H)\rightarrow\Omega^{1}(X)$ denote the operator
\begin{equation}\label{Ddef}
D^{1}:=\mathcal{L}_{\xi}+d_{H}\star_{H} d_{H}
\end{equation}
\end{define}
\noindent
and observe that we can also write $S(A)=\langle A,-\star_{H}D^{1}A\rangle_{\kappa}^{1}$, identifying $S(A)$ with the operator $-\star_{H}D^{1}$ on $\Omega^{1}(H)$.  Thus, we have proven the following
\begin{prop}\label{prop1}
The new action, $S(A)$, as defined in Eq. \ref{newact}, for the ``shifted'' partition function of Eq. \ref{Anom1} can be expressed as a quadratic form on the space of horizontal forms $\Omega^{1}(H)$ as follows:
\begin{equation}
S(A)= \langle A, -\star DA\rangle
\end{equation}
or equivalently as,
\begin{equation}
S(A)=\langle A,-\star_{H}D^{1}A\rangle_{\kappa}^{1}
\end{equation}
where $D$ and $D^{1}$ are the second order operators defined in Eq.'s \ref{Ddef1} and \ref{Ddef}, respectively.  $\langle \cdot, \cdot\rangle$ is the usual $L^{2}$ inner product on $\Omega^{1}(X)$, and $\langle \cdot, \cdot\rangle^{1}_{\kappa}$ is defined in Eq. \ref{inner}.
\end{prop}

\section{Gauge group and the isotropy subgroup}\label{gsec}
In order to extract anything mathematically meaningful out of this construction we will need to divide out the action of the gauge group $\mathcal{G}_{P}$ on $\mathcal{A}_{P}$.  At this point we observe that the gauge group $\mathcal{G}_{P}\simeq \text{Maps}(X\rightarrow U(1))$ naturally descends to a ``horizontal'' action on $\bar{\mathcal{A}}_{P}$, which infinitesimally can be written as:
\begin{equation}\label{action}
\theta\in \text{Lie}(\mathcal{G}_{P}):A\mapsto A+d_{H}\theta
\end{equation}
Following \cite{s2}, we let $H_{A}$ denote the isotropy subgroup of $\mathcal{G}_{P}$ at a point $A\in\bar{\mathcal{A}}_{P}$.  Note that $H_{A}$ can be canonically identified for every $A\in\bar{\mathcal{A}}_{P}$, and so we simply write $H$ for the isotropy group.  The condition for an element of the gauge group $h(x)=e^{i\theta(x)}$ to be in the isotropy group is that $d_{H}\theta=0$, given definition \ref{action} above.  By (\cite{r} ; Prop. 12), we see that the condition $d_{H}\theta=0$ implies that $\theta$ is harmonic, and so $\mathcal{L}_{\xi}\theta=0$.  Therefore we have $d\theta=0$ since $d=d_{H}+\kappa\wedge \mathcal{L}_{\xi}$.  Thus, the group $H$ can be identified with the group of constant maps from $X$ into $U(1)$; hence, is isomorphic to $U(1)$.  We let $Vol(H)$ denote the volume of the isotropy subgroup, computed with respect to the metric induced from $\mathcal{G}_{P}$, so that
\begin{equation}\label{volu}
Vol(H)=\left[\int_{X}\kappa\wedge d\kappa\right]^{1/2}=\left[ n+\sum_{j=1}^{N}\frac{\beta_{j}}{\alpha_{j}}\right]^{1/2}
\end{equation}
where $[n; (\alpha_{1},\beta_{1}),\ldots,(\alpha_{N},\beta_{N})]$ are the Seifert invariants of our Seifert manifold $X$.  The last equality in Eq. \ref{volu} above follows from Eq. 3.22 of \cite{bw}.

\section{The partition function}\label{partsec}
We now have
\begin{eqnarray}
Z_{U(1)}(X,p,k)&=&\frac{e^{\pi i k S_{X,P}(A_{0})}}{Vol(\mathcal{G}_{P})}\int_{\bar{\mathcal{A}}_{P}}\bar{D}A\,\, e^{\left[\frac{i k}{4\pi}S(A)\right]}\nonumber\\
               &=&\frac{Vol(\mathcal{G}_{P})}{Vol(H)}\frac{e^{\pi i k S_{X,P}(A_{0})}}{Vol(\mathcal{G}_{P})}\int_{\bar{\mathcal{A}}_{P}/\mathcal{G}_{P}}\,\, e^{\left[\frac{i k}{4\pi}S(A)\right]}\left[det'(d_{H}^{*}d_{H})\right]^{1/2}\,\,\mu\nonumber\\
               &=&\frac{e^{\pi i k S_{X,P}(A_{0})}}{Vol(H)}\int_{\bar{\mathcal{A}}_{P}/\mathcal{G}_{P}}\,\, e^{\left[\frac{i k}{4\pi}S(A)\right]}\left[det'(d_{H}^{*}d_{H})\right]^{1/2}\,\,\mu\label{oscil}
\end{eqnarray}
where $\mu$ is the induced measure on the quotient space $\bar{\mathcal{A}}_{P}/\mathcal{G}_{P}$ and $det'$ denotes a regularized determinant to be defined later.  Since $S(A)=\langle A,-\star_{H}D^{1}A\rangle_{\kappa}^{1}$ is quadratic in $A$, we may apply the method of stationary phase (\cite{s1}, \cite{gs}) to evaluate the oscillatory integral (\ref{oscil}) exactly.  We obtain,
\begin{eqnarray}\label{intzeta}
&&\\
Z_{U(1)}(X,p,k)&=&\frac{e^{\pi i k S_{X,P}(A_{0})}}{Vol(H)}\int_{\mathcal{M}_{P}}\,\, e^{\frac{\pi i}{4}\,\,sgn(-\star_{H}D^{1})}\frac{\left[det'(d_{H}^{*}d_{H})\right]^{1/2}}{\left[det'(-k\star_{H}D^{1})\right]^{1/2}}\,\,\nu\nonumber
\end{eqnarray}
where $\mathcal{M}_{P}$ denotes the moduli space of flat connections modulo the gauge group and $\nu$ denotes the induced measure on this space.  Note that we have included a factor of $k$ in our regularized determinant since this factor occurs in the exponent multiplying $S(A)$.

\section{Zeta function determinants}
We will use the following to define the regularized determinant of $-k\star_{H}D^{1}$
\begin{prop}\label{product}\cite{s2}
Let $\mathcal{H}_{0}$, $\mathcal{H}_{1}$ be Hilbert spaces, and $S:\mathcal{H}_{1}\rightarrow \mathcal{H}_{1}$ and $T:\mathcal{H}_{0}\rightarrow \mathcal{H}_{1}$ such that $S^{2}$ and $TT^{*}$ have well defined zeta functions with discrete spectra and meromorphic extensions to $\C$ that are regular at 0 (with at most simple poles on some discrete subset).  If $ST=0$, and $S^{2}$ is self-adjoint, then
\begin{equation}
det'(S^{2}+TT^{*})=det'(S^{2})det'(TT^{*})
\end{equation}
\end{prop}
\begin{proof}
This equality follows from the facts that $S^{2}TT^{*}=0$ and $TT^{*}S^{2}=0$ (i.e. these operators commute), which both follow from $ST=0$ and the fact that $S^{2}$ and $TT^{*}$ are both self-adjoint.
\end{proof}
\noindent
Following the notation of Eq.'s (3)-(6) in section 2 of \cite{s2}, we set the operators $S=-k\star_{H}D^{1}$ and $T=k d_{H}d_{H}^{*}$ on $\Omega^{1}(H)$ and observe that $ST=0$ since (\ref{complex}) is a complex.
With Prop. \ref{product} as \emph{motivation}, we make the formal definition
%\Label{TANAKA-WEBSTER DEPENDENCE}
%\mute{We need to define the k-dependence to be independent of the metric here.  We should multiply by the corresponding factor.}{text2}
%
%
%
%
%
\begin{equation}\label{regdet}
det'(-k\star_{H}D^{1}):=C(k,J)\cdot\frac{[det'(S^{2}+TT^{*})]^{1/2}}{[det'(TT^{*})]^{1/2}}
\end{equation}
where $S^{2}+TT^{*}=k^{2}((D^{1})^{*}D^{1}+(d_{H}d_{H}^{*})^{2})$, $TT^{*}=k^{2}(d_{H}d_{H}^{*})^{2}$ and
\begin{equation}
C(k,J):=k^{\left(-\frac{1}{1024}\int_{X}R^{2}\,\kappa\wedge d\kappa\right)}
\end{equation}
is a function of $R\in C^{\infty}(X)$, the Tanaka-Webster scalar curvature of $X$, which in turn depends only on a choice of a compatible complex structure $J\in \text{End}(H)$.  That is, given a choice of contact form $\kappa\in\Omega^{1}(X)$, the choice of complex structure $J\in \text{End}(H)$ determines uniquely an associated metric.  We have defined $det'(-k\star_{H}D^{1})$ in this way to eliminate the metric dependence that would otherwise occur in the $k$-dependence of this determinant.  The motivation for the definition of the factor $C(k,J)$ comes explicitly from Prop. \ref{Jdepend} below.\\
\\
\noindent
%By definition $[det'(TT^{*})]^{1/2}=[det'k^{2}(d_{H}d_{H}^{*})^{2}]^{1/2}=det'(k d_{H}d_{H}^{*})$.
The operator
\begin{equation}\label{maxLap}
\Delta:=(D^{1})^{*}D^{1}+(d_{H}d_{H}^{*})^{2}
\end{equation}
is actually equal to the middle degree Laplacian defined in Eq. (10) of \cite{rs} and has some nice analytic properties.  In particular, it is maximally hypoelliptic and invertible in the Heisenberg symbolic calculus (See \cite{rs} ; \S 3.1).  We define the regularized determinant of $\Delta$ via its zeta function (\cite{rs} ; Pg. 10)
\begin{equation}
\zeta(\Delta)(s):=\sum_{\lambda\in\text{spec}^{*}(\Delta)}\lambda^{-s}
\end{equation}
Note that our definition agrees with \cite{rs} up to a constant term $\text{dim}H^{1}(X,D)$, which is finite by hypoellipticity (\cite{rs} ; Pg. 11).  Also, $\zeta(\Delta)(s)$ admits a meromorphic extension to $\C$ that is regular at $s=0$ (\cite{p2} ; \S 4).  Thus, we define the regularized determinant of $\Delta$ as
\begin{equation}
det'(\Delta):=e^{-\zeta'(\Delta)(0)}
\end{equation}
%As observed in \cite{s2}, the non-zero eigenvalues of $d_{H}d_{H}^{*}$ and $d_{H}^{*}d_{H}$ coincide and therefore %$det'(d_{H}d_{H}^{*})=det'(d_{H}^{*}d_{H})$.
Let $\Delta_{0}:=(d_{H}^{*}d_{H})^{2}$ on $\Omega^{0}(X)$, $\Delta_{1}:=\Delta$ on $\Omega^{1}(H)$ and define $\zeta_{i}(s):=\zeta(\Delta_{i})(s)$.  We claim the following\\
\begin{prop}
For any real number $0<c\in\R$,
\begin{equation}\label{delta}
det'(c\Delta_{i}):=c^{\zeta_{i}(0)}det'(\Delta_{i})
\end{equation}
for $i=0,1$.
\end{prop}
\begin{proof}
To prove this claim, recall that $\zeta_{i}(s)=\zeta(\Delta_{i})(s)$ for $i=0,1$, scale as follows:
\begin{equation}\label{unsure}
\zeta(c\Delta_{i})(s)=c^{-s}\zeta(\Delta_{i})(s).
\end{equation}
From here we simply calculate the scaling of the regularized determinants using the definition
\begin{equation}
det'(\Delta_{i}):=e^{-\zeta'(\Delta_{i})(0)}
\end{equation}
and the claim is proven.
\end{proof}
\noindent
The following will be useful.
\begin{prop}\label{Jdepend}
For $\Delta_{0}:=(d_{H}^{*}d_{H})^{2}$ on $\Omega^{0}(X)$, $\Delta_{1}:=\Delta$ on $\Omega^{1}(H)$ defined as above and $\zeta_{i}(s):=\zeta(\Delta_{i})(s)$, we have
\begin{eqnarray}
&&\\\label{dimen}
\zeta_{0}(0)-\zeta_{1}(0)&=&\left(-\frac{1}{512}\int_{X}R^{2}\,\kappa\wedge d\kappa\right)+\text{dim Ker}\Delta_{1}-\text{dim Ker}\Delta_{0}\nonumber\\
                         &=&\left(-\frac{1}{512}\int_{X}R^{2}\,\kappa\wedge d\kappa\right)+\text{dim} H^{1}(X,d_{H})-\text{dim} H^{0}(X,d_{H}).
\end{eqnarray}
where $R\in C^{\infty}(X)$ is the Tanaka-Webster scalar curvature of $X$ and $\kappa\in\Omega^{1}(X)$ is our chosen contact form as usual.
\end{prop}
\begin{proof}
Let
\begin{eqnarray*}
\hat{\zeta_{0}}(s)&:=&\text{dim Ker}\Delta_{0}+\zeta_{0}(s)\\
\hat{\zeta_{1}}(s)&:=&\text{dim Ker}\Delta_{1}+\zeta_{1}(s)
\end{eqnarray*}
denote the zeta functions as defined in \cite{rs}.  From (\cite{rs} ; Cor. 3.8), one has that
\begin{equation*}
\hat{\zeta_{1}}(0)=2\hat{\zeta_{0}}(0)
\end{equation*}
for all 3-dimensional contact manifolds.  By (\cite{bhr} ; Theorem 8.8), one knows that on CR-Seifert manifolds that
\begin{equation*}
\hat{\zeta_{0}}(0)=\hat{\zeta}(\Delta_{0})(0)=\hat{\zeta}(\Delta_{0}^{2})(0)=\frac{1}{512}\int_{X}R^{2}\,\kappa\wedge d\kappa
\end{equation*}
Thus,
\begin{equation*}
\hat{\zeta_{1}}(0)=\frac{1}{256}\int_{X}R^{2}\,\kappa\wedge d\kappa
\end{equation*}
By our definition of the zeta functions, which differ from that of \cite{rs} by constant dimensional terms, we therefore have
\begin{eqnarray*}
\zeta_{0}(0)&=&\frac{1}{512}\int_{X}R^{2}\,\kappa\wedge d\kappa-\text{dim Ker}\Delta_{0}\\
\zeta_{1}(0)&=&\frac{1}{256}\int_{X}R^{2}\,\kappa\wedge d\kappa-\text{dim Ker}\Delta_{1}
\end{eqnarray*}
Hence,
\begin{eqnarray*}
\zeta_{0}(0)-\zeta_{1}(0)&=&\left[\frac{1}{512}\int_{X}R^{2}\,\kappa\wedge d\kappa-\text{dim Ker}\Delta_{0}\right]-\left[\frac{1}{256}\int_{X}R^{2}\,\kappa\wedge d\kappa-\text{dim Ker}\Delta_{1}\right]\\
                         &=&\left(-\frac{1}{512}\int_{X}R^{2}\,\kappa\wedge d\kappa\right)+\text{dim Ker}\Delta_{1}-\text{dim Ker}\Delta_{0}\\
                         &=&\left(-\frac{1}{512}\int_{X}R^{2}\,\kappa\wedge d\kappa\right)+\text{dim} H^{1}(X,d_{H})-\text{dim} H^{0}(X,d_{H}).
\end{eqnarray*}
and the result is proven.
\end{proof}
\noindent
We now have the following
\begin{prop}\label{rigdet}
The term inside of the integral of Eq. \ref{intzeta} has the following expression in terms of the hypoelliptic Laplacians, $\Delta_{0}$ and $\Delta_{1}$, as defined in Prop. \ref{Jdepend}:
\begin{equation}
\frac{\left[det'(d_{H}^{*}d_{H})\right]^{1/2}}{\left[det'(-k\star_{H}D^{1})\right]^{1/2}}=k^{n_{X}}\frac{[det'(\Delta_{0})]^{1/2}}{\left[det'(\Delta_{1})\right]^{1/4}}
\end{equation}
where
\begin{equation}\label{nX}
n_{X}:=\frac{1}{2}(\text{dim} H^{1}(X,d_{H})-\text{dim} H^{0}(X,d_{H})).
\end{equation}
\end{prop}
\begin{proof}
\begin{eqnarray}\label{cal1}
&&\\
\frac{\left[det'(d_{H}^{*}d_{H})\right]^{1/2}}{\left[det'(-k\star_{H}D_{\kappa}^{1})\right]^{1/2}}&=&C(k,J)^{-1}\cdot\frac{\left[det'(d_{H}^{*}d_{H})^{2}\right]^{1/4}\cdot \left[det'k^{2} (d_{H}d_{H}^{*})^{2}\right]^{1/4}}{\left[det'(k^{2}\Delta)\right]^{1/4}}\nonumber\\\label{cal2}
                    &=&C(k,J)^{-1}\cdot\frac{k^{\zeta_{0}(0)/2}\left[det'(\Delta_{0})\right]^{1/4}\cdot \left[det'(\Delta_{0})\right]^{1/4}}{k^{\zeta_{1}(0)/2}\left[det'(\Delta_{1})\right]^{1/4}}\\
                    &=&C(k,J)^{-1}\cdot k^{\frac{1}{2}(\zeta_{0}(0)-\zeta_{1}(0))}\frac{[det'(\Delta_{0})]^{1/2}}{\left[det'(\Delta_{1})\right]^{1/4}}\nonumber\\
                    &=&C(k,J)^{-1}\cdot C(k,J)\cdot k^{n_{X}}\frac{[det'(\Delta_{0})]^{1/2}}{\left[det'(\Delta_{1})\right]^{1/4}},\,\text{Prop. \ref{Jdepend}},\nonumber\\
                    &=&k^{n_{X}}\frac{[det'(\Delta_{0})]^{1/2}}{\left[det'(\Delta_{1})\right]^{1/4}}\nonumber
\end{eqnarray}
\noindent
where the second last line comes from Eq. \ref{dimen}.  Also note that $d_{H}^{*}d_{H}$ and $d_{H}d_{H}^{*}$ have the same eigenvalues (by standard arguments), which allows us to proceed to Eq. \ref{cal2} from Eq. \ref{cal1}.
\end{proof}
\begin{rem}\label{rmknX}
Note that by (\cite{rs} ; Prop. 2.2), the definition of $n_{X}$ (see Eq. \ref{nX}) here is exactly equal to the quantity $m_{X}:=\frac{1}{2}(\text{dim} H^{1}(X,d)-\text{dim} H^{0}(X,d))$ of (\cite{m} ; Eq. 5.18).  This shows that our partition function has the same $k$-dependence as that in \cite{m}.
\end{rem}
\section{The eta invariant}\label{esec}
Next we regularize the signature $sgn(-\star_{H}D^{1})$ via the eta-invariant and set $sgn(-\star_{H}D^{1})=\eta(-\star_{H}D^{1})(0):=\eta(-\star_{H}D^{1})$ where
\begin{equation}
\eta(-\star_{H}D^{1})(s):=\sum_{\lambda\in\text{spec}^{*}(-\star_{H}D^{1})}(sgn\lambda)|\lambda|^{-s}
\end{equation}
\noindent
Finally, we may now write the result for our partition function
\begin{eqnarray}\label{ctorsion}
&&\\
Z_{U(1)}(X,p,k)&=&k^{n_X}e^{\pi i k S_{X,P}(A_{0})}e^{\frac{\pi i}{4}\eta(-\star_{H}D^{1})}\int_{\mathcal{M}_{P}}\,\, \frac{1}{Vol(H)}\frac{[det'(\Delta_{0})]^{1/2}}{\left[det'(\Delta_{1})\right]^{1/4}}\,\,\nu\nonumber
\end{eqnarray}
where $n_{X}:=\frac{1}{2}(\text{dim} H^{1}(X,d_{H})-\text{dim} H^{0}(X,d_{H}))$.  Note that $\nu$ is a measure on $\mathcal{M}_{P}$ (the moduli space of flat connections modulo the gauge group) relative to the horizontal structure on the tangent space of $\mathcal{M}_{P}$.\\
\\
\section{Torsion}\label{tsec}
Now we will study the quantity $\frac{1}{Vol(H)}\frac{[det'(\Delta_{0})]^{1/2}}{\left[det'(\Delta_{1})\right]^{1/4}}\,\,\nu$ inside of the integral in Eq. \ref{ctorsion}, and in particular how it is related to the analytic contact torsion $T_{C}$.  First, recall that (\cite{rs};Eq. 16)
\begin{equation}\label{torsion}
T_{C}:=\text{exp}\left( \frac{1}{4}\sum_{q=0}^{3}(-1)^{q}w(q)\zeta'(\Delta_{q})(0)\right)
\end{equation}
where
\begin{equation}
w(q)=
\begin{cases} q & \text{if $q\leq 1$,}
\\
q+1 &\text{if $q>1$.}
\end{cases}
\end{equation}
in the case where $\text{dim}(X)=3$.  Note that we have chosen a sign convention that leads to the inverse of the definition of $T_C$ in \cite{rs}.  Recall (\cite{rs}, Eq. 10),
\begin{equation}\label{conlapl}
\Delta_{q}=
\begin{cases} (d_{H}^{*}d_{H}+d_{H}d_{H}^{*})^{2} & \text{if $q = 0,3$,}
\\
D^{*}D+(d_{H}d_{H}^{*})^{2} &\text{if $q=1$.}
\\
DD^{*}+(d_{H}^{*}d_{H})^{2} &\text{if $q=2$.}
\end{cases}
\end{equation}
We would, however, like to work with torsion when viewed as a density on the determinant line
\begin{eqnarray*}
|\text{det}H^{\bullet}(X,d_{H})^{*}|&:=&|\text{det}H^{0}(X,d_{H})|\otimes|\text{det}H^{1}(X,d_{H})^{*}|\\
                                    &\otimes&|\text{det}H^{2}(X,d_{H})|\otimes|\text{det}H^{3}(X,d_{H})^{*}|
\end{eqnarray*}
We follow \cite{rsi} and \cite{m} and make the analogous definition.
\begin{define}\label{torsdef}
Define the analytic torsion as a density as follows
\begin{equation*}
T^{d}_{C}:=T_{C}\cdot\delta_{|\text{det}H^{\bullet}(X,d_{H})|}
\end{equation*}
where $T_{C}$ is as defined in Eq. \ref{torsion}, and
\begin{equation*}
\delta_{|\text{det}H^{\bullet}(X,d_{H})|}:=\otimes_{q=0}^{dim X}|\nu_{1}^{q}\wedge\cdots\wedge \nu_{b_{q}}^{q}|^{(-1)^{q}}
\end{equation*}
where $\{\nu_{1}^{q},\cdots ,\nu_{b_{q}}^{q}\}$ is an orthonormal basis for the space of harmonic contact forms $\mathcal{H}^{q}(X,d_{H})$ with the inner product defined in Eq. \ref{inner}.  Note that $\mathcal{H}^{q}(X,d_{H})$ is canonically identified with the cohomology space $H^{q}(X,d_{H})$, and   $b_{q}:=\text{dim}(H^{q}(X,d_{H}))$ is
the $q^{th}$ contact Betti number.
\end{define}
\noindent
Let
\begin{equation*}
\nu^{(q)}:=\nu_{1}^{q}\wedge\cdots\wedge \nu_{b_{q}}^{q}
\end{equation*}
and write the analytic torsion of a compact connected Seifert 3-manifold $X$ as
\begin{eqnarray}
T^{d}_{C}=T_{C}\times |\nu^{(0)}|\otimes|\nu^{(1)}|^{-1}\otimes|\nu^{(2)}|\otimes|\nu^{(3)}|^{-1}.
\end{eqnarray}
In terms of regularized determinants, we have
\begin{equation}
T_{C}=\left[(det'(\Delta_{0}))^{0}\cdot (det'(\Delta_{1}))^{1}\cdot (det'(\Delta_{2}))^{-3}\cdot (det'(\Delta_{3}))^{4}\right]^{1/4}
\end{equation}
where $\Delta_{q}$, $0\leq q\leq 3$, denotes the Laplacians on the contact complex
as defined in (\cite{rs} ; Eq. 10) and recalled in Eq. \ref{conlapl} above.  This notation agrees with our notation for $\Delta_{0}$, $\Delta_{1}$ as in Eq. \ref{delta}.  The Hodge $\star$-operator induces the equivalences $\Delta_{q}\simeq \Delta_{3-q}$ (see \cite{rs};Theorem 3.4) and allows us to write
\begin{eqnarray}
T_{C}&=&\left[(det'(\Delta_{0}))^{0}\cdot (det'(\Delta_{1}))^{1}\cdot (det'(\Delta_{2}))^{-3}\cdot (det'(\Delta_{3}))^{4}\right]^{1/4}\\\label{eq3}
     &=&\frac{det'(\Delta_{0})}{(det'(\Delta_{1}))^{1/2}}
\end{eqnarray}
Also, from the isomorphisms $H^{q}(X,\R)\simeq H^{q}(X,d_{H})$ of Prop. 2.2 of \cite{rs}, we have Poincar\'{e} duality $H^{q}(X,d_{H})\simeq H^{3-q}(X,d_{H})^{*}$, and therefore
\begin{equation}\label{eqzo}
T^{d}_{C}=T_{C}\times |\nu^{0}|^{\otimes 2}\otimes(|\nu^{1}|^{-1})^{\otimes 2}
\end{equation}
Moreover, by \cite{r} ( Prop. 12), $\mathcal{H}^{q}(X,d_{H})=\mathcal{H}^{q}(X,\R)$, and thus any orthonormal basis $\nu^{(0)}$ of $\mathcal{H}^{0}(X,d_{H})\simeq\R$ is a constant such that
\begin{equation}\label{eq1}
|\nu^{(0)}|=\left[\int_{X}\kappa\wedge d\kappa\right]^{-1/2}
\end{equation}
Also, recall that the tangent space $T_{A}\mathcal{M}_{P}\simeq H^{1}(X,d_{H}) \simeq H^{1}(X,\R)$, at any point $A\in\mathcal{M}_{P}$.  The measure $\nu$ on $\mathcal{M}_{P}$ that occurs in Eq. \ref{ctorsion} is defined relative to the metric on $H^{1}(X,d_{H})\simeq\mathcal{H}^{1}(X,d_{H})$, which can be identified with the usual $L^{2}$-metric on forms.    Thus
the measure $\nu$ may be identified with the inverse of the density $|\nu^{(1)}|$ by dualizing the orthogonal basis $\{\nu_{1}^{1}, \ldots, \nu_{b_{1}}^{1}\}$ for $\mathcal{H}^{1}(X,d_{H})$; i.e.
\begin{equation}\label{eq2}
\nu=|\nu^{(1)}|^{-1}=|\nu_{1}^{1}\wedge \cdots \wedge \nu_{b_{1}}^{1}|^{-1}
\end{equation}
Putting together equations \ref{eq3}, \ref{eq1}, \ref{eq2} into equation \ref{eqzo}, we have
\begin{eqnarray}
T^{d}_{C}&=&T_{C}\times |\nu^{0}|^{\otimes 2}\otimes(|\nu^{1}|^{-1})^{\otimes 2}\\
         &=&\frac{det'(\Delta_{0})}{(det'(\Delta_{1}))^{1/2}}\cdot\left[\int_{X}\kappa\wedge d\kappa\right]^{-1} \nu^{\otimes 2}\\
         &=&\text{Vol}(H)^{-2}\frac{det'(\Delta_{0})}{(det'(\Delta_{1}))^{1/2}}\cdot\nu^{\otimes 2}
\end{eqnarray}
We have thus proven the following,
\begin{prop}\label{tprop}
The contact analytic torsion, when viewed as a density $T^{d}_{C}$ as in definition \ref{torsdef}, can be identified as follows:
\begin{equation}
(T^{d}_{C})^{1/2}=\frac{1}{Vol(H)}\frac{[det'(\Delta_{0})]^{1/2}}{\left[det'(\Delta_{1})\right]^{1/4}}\,\,\nu
\end{equation}
\end{prop}
Our partition function is now
\begin{equation}\label{newpar}
\bar{Z}_{U(1)}(X,p,k)=k^{n_X}e^{\pi i k S_{X,P}(A_{0})}e^{\frac{\pi i}{4}\eta(-\star_{H}D^{1})}\int_{\mathcal{M}_{P}}\,\, (T^{d}_{C})^{1/2}
\end{equation}
This partition function should be completely equivalent to the partition function defined in (\cite{m} ; Eq. 7.27):
\begin{equation}\label{oldpar}
Z_{U(1)}(X,p,k)=k^{m_X}e^{\pi i k S_{X,P}(A_{0})}e^{\frac{\pi i}{4}\eta(-\star d)}\int_{\mathcal{M}_{P}}\,\, (T^{d}_{RS})^{1/2}.
\end{equation}
Our goal in the remainder is to show that this is indeed the case.  Our first observation is that $(T^{d}_{C})^{1/2}$ is equal to the Ray-Singer torsion $(T^{d}_{RS})^{1/2}$ that occurs in (\cite{m} ; Eq. 7.27).  This follows directly from (\cite{rs} ; Theorem 4.2); note that their sign convention makes $T_C$ the inverse of our definition.

\section{Regularizing the eta-invariants}\label{fsec}
Since we have seen that our $k$-dependence matches that in \cite{m} (i.e. $m_{X}=n_{X}$ ; cf. Remark \ref{rmknX}), the only thing left to do is to reconcile the eta invariants, $\eta(-\star_{H}D^{1})$ and $\eta(-\star d)$.  As observed in \cite{w3}, the correct quantity to compare our eta invariant to would be
\begin{equation}\label{reg1}
\frac{\eta(-\star d)}{4}+\frac{1}{12}\frac{\text{CS}(A^{g})}{2\pi}.
\end{equation}
where,
\begin{equation}
\text{CS}(A^{g})=\frac{1}{4\pi}\int_{X}Tr(A^{g}\wedge dA^{g}+\frac{2}{3} A^{g}\wedge A^{g}\wedge A^{g})
\end{equation}
is the gravitational Chern-Simons term, with $A^{g}$ the Levi-Civita connection on the spin bundle of $X$ for a given metric $g$ on $X$.  See Appendix \ref{appen2} for a short exposition on the regularization of $\eta(-\star d)$ in Eq. \ref{reg1}.  It was noticed in \cite{w3} that in the quasi-classical limit, quantum anomalies can occur that can break topological invariance.  Invariance may be restored in this case only after adding a counterterm to the eta invariant.  Our job then is to perform a similar analysis for the eta invariant $\eta(-\star_{H}D^{1})$, which depends on a choice of metric.  Of course, our choice of metric is natural in this setting and is adapted to the contact structure.  One possible approach is to consider variations over the space of such natural metrics and calculate the corresponding variation of the eta invariant, giving us a local formula for the counterterm that needs to be added.  Such a program has already been initiated in \cite{bhr}.\\
\\
Our starting point is the conjectured equivalence that results from the identification of Eq.'s \ref{newpar} and \ref{oldpar}:
\begin{equation}\label{etavar}
e^{\pi i\left[\frac{\eta(-\star d)}{4}+\frac{1}{12}\frac{\text{CS}(A^{g})}{2\pi}\right]}\text{``$=$''}e^{\frac{\pi i}{4}\left[\eta(-\star_{H} D^{1})+C_{T}\right]}
\end{equation}
where $C_{T}$ is some appropriate counterterm that yields an invariant comparable to the left hand of this equation.  As noted in Appendix \ref{appen2}, the left hand side of this equation depends on a choice of $2$-framing on $X$, and since we have a rule (cf. Eq. \ref{partform}) for how the partition function transforms when the framing is twisted, we basically have a topological invariant.  Alternatively, as also noted in Appendix \ref{appen2}, one can use the main result of \cite{at} and fix the canonical $2$-framing on $TX\oplus TX$.  We therefore expect the same type of phenomenon for the right hand side of this equation, having at most a $\Z$-dependence on the regularization of our eta invariant, along with a rule that tells us how the partition function changes when our discrete invariants are ``twisted,'' once again yielding a topological invariant.\\
\\
Let us first make the statement of the conjecture of Eq. \ref{etavar} more precise.  We should have the following
\begin{conj}
$(X,\phi,\xi,\kappa,g)$ a closed quasi-regular K-contact three-manifold.  Then there exists a counterterm, $C_{T}$, such that
$$e^{\frac{\pi i}{4}\left[\eta(-\star_{H} D^{1})+C_{T}\right]}$$ is a topological invariant that is identically equal to the topological invariant $$e^{\pi i\left[\frac{\eta(-\star d)}{4}+\frac{1}{12}\frac{\text{CS}(A^{g})}{2\pi}\right]},$$ where $\text{CS}(A^{g})$ and all relevant operators are defined with respect to the metric $g$ on $X$ and we use the canonical 2-framing \cite{at}.
\end{conj}
Our regularization procedure for $\eta(-\star_{H} D^{1})$ will be quite different than that used for $\eta(-\star d)$.  Since we are restricted to a class of metrics that are compatible with our contact structure, we are really only concerned with finding appropriate counterterms for $\eta(-\star_{H} D^{1})$ that will eliminate our dependence on the choice of contact form $\kappa$ and complex structure $J\in\text{End}(H)$.  In the case of interest, we observe that our regularization may be obtained in one stroke by introducing the \emph{renormalized $\eta$-invariant}, $\eta_{0}(X,\kappa)$, of $X$ that is discussed in (\cite{bhr} ; \S 3).  Before giving the definition of $\eta_{0}(X,\kappa)$, we require the following
\begin{lem}(\cite{bhr} ; Lemma 3.1)
Let $(X,J,\kappa)$ be a strictly pseudoconvex pseudohermitian 3-manifold.  The $\eta$-invariants of the family of metrics $g_{\epsilon}:=\epsilon^{-1}\kappa\otimes\kappa+d\kappa(\cdot,J\cdot)$ have a decomposition in homogeneous terms:
\begin{equation}\label{etalem}
\eta(g_{\epsilon})=\sum_{i=-2}^{2}\eta_{i}(X,\kappa)\epsilon^{i}.
\end{equation}
The terms $\eta_{i}$ for $i\neq 0$ are integrals of local pseudohermitian invariants of $(X,\kappa)$, and the $\eta_{i}$ for $i>0$ vanish when the Tanaka-Webster torsion, $\tau$, vanishes.
\end{lem}
We then make the following
\begin{define}
Let $(X,\kappa)$ be a compact strictly pseudoconvex pseudohermitian 3-dimensional manifold.  The \emph{renormalized $\eta$-invariant}
$\eta_{0}(X,\kappa)$ of $(X,\kappa)$ is the constant term in the expansion of Eq. \ref{etalem} for the $\eta$-invariants of the family of metrics $g_{\epsilon}:=\epsilon^{-1}\kappa\otimes\kappa+d\kappa(\cdot,J\cdot)$.
\end{define}
Our assumption that $X$ is K-contact ensures that the Reeb flow preserves the metric.  In this situation, it is known that the Tanaka-Webster torsion necessarily vanishes (cf. \cite{bhr} ; \S 3).  In the case where the torsion of $(X,\kappa)$ vanishes, the terms $\eta_{i}(X,\kappa)$ in Eq. \ref{etalem} vanish for $i>0$, so that when $\epsilon\rightarrow \infty$, one has
\begin{equation}\label{vantor}
\eta_{0}(X,\kappa)=\lim_{\epsilon\rightarrow\infty}\eta(g_{\epsilon}):=\eta_{ad}
\end{equation}
The limit $\eta_{ad}$ is known as the \emph{adiabatic limit} and has been studied in \cite{bc} and \cite{dai}, for example.  The adiabatic limit is the case where the limit is taken as $\epsilon$ goes to infinity,
\begin{equation}
\eta_{ad}:=\lim_{\epsilon\rightarrow\infty}\eta(g_{\epsilon}),
\end{equation}
while the the renormalized $\eta$-invariant, $\eta_{0}(X,\kappa)$, is naturally interpreted as the constant term in the asymptotic expansion for $(\eta(g_{\epsilon}))$ in powers of $\epsilon$, when $\epsilon$ goes to $0$.  This reverse process of taking $\epsilon$ to $0$ is also known as the \emph{diabatic} limit.  When torsion vanishes (i.e. when the Reeb flow preserves the metric), Eq. \ref{vantor} is the statement that the diabatic and adiabatic limits agree.  One of the main challenges for our future work will be to extend beyond the case where torsion vanishes.  This will naturally involve the study of the diabatic limit.  For now, we are restricted to the case of vanishing torsion.  In this case, the main result that we will use is the following
\begin{thm}\label{eta0thm}(\cite{bhr} ; Theorem 1.4)
Let $X$ be a compact CR-Seifert 3-manifold, with $U(1)$-action generated by the Reeb field of an $U(1)$-invariant contact form $\kappa$.  If $R$ is the Tanaka-Webster curvature of $(X,\kappa)$ and $D^{1}$ is the middle degree operator of the contact complex (cf. Eq. \ref{complex} and \ref{Ddef}), then
\begin{equation}
\eta_{0}(X,\kappa)=\eta(-\star_{H}D^{1})+\frac{1}{512}\int_{X}R^{2}\,\,\kappa\wedge d\kappa.
\end{equation}
\end{thm}
Theorem \ref{eta0thm} compels us to conjecture that $C_{T}=\frac{1}{512}\int_{X}R^{2}\,\,\kappa\wedge d\kappa$.  Our motivation for this comes from the fact that $\eta_{0}(X,\kappa)$ is a topological invariant in our case.  We have the following,
\begin{thm}\label{topthm}(\cite{bhr} ; cf. Remark 9.6 and Eq. 27)
If $X$ is a CR-Seifert manifold, then $\eta_{0}(X,\kappa)$ is a topological invariant and
\begin{equation}
\eta_{0}(X,\kappa)=1+\frac{d}{3}+4\sum_{j=1}^{N}s(\alpha_{j},\beta_{j}),
\end{equation}
where $d\in\Q$ is the degree of $X$ as a compact $U(1)$-orbifold bundle and
\begin{equation}
s(\alpha,\beta):=\frac{1}{4\alpha}\sum_{k=1}^{\alpha-1}cot\left(\frac{\pi k}{\alpha}\right)cot\left(\frac{\pi k\beta}{\alpha}\right)
\end{equation}
is the classical Rademacher-Dedekind sum, where $[n; (\alpha_{1},\beta_{1}),\ldots,(\alpha_{N},\beta_{N})]$ (for $(\alpha_{i},\beta_{i})=1$ relatively prime) are the Seifert invariants of $X$.
\end{thm}
Thus, we are led to consider the natural topological invariant $e^{\frac{\pi i}{4}\left[\eta_{0}(X,\kappa)\right]}$ and how it compares with the topological invariant $e^{\pi i\left[\frac{\eta(-\star d)}{4}+\frac{1}{12}\frac{\text{CS}(A^{g})}{2\pi}\right]}$.  We consider the limit
\begin{eqnarray}
\lim_{\epsilon\rightarrow\infty}e^{\pi i\left[\frac{\eta(-\star_{\epsilon} d)}{4}+\frac{1}{12}\frac{\text{CS}(A^{g_{\epsilon}})}{2\pi}\right]}
\end{eqnarray}
where $g_{\epsilon}=\epsilon^{-1}\kappa\otimes\kappa+d\kappa(\cdot,J\cdot)$ is the natural metric associated to $X$.  On the one hand, since this is a topological invariant, and is independent of the metric, we must have
\begin{eqnarray}
\lim_{\epsilon\rightarrow\infty}e^{\pi i\left[\frac{\eta(-\star_{\epsilon} d)}{4}+\frac{1}{12}\frac{\text{CS}(A^{g_{\epsilon}})}{2\pi}\right]}=e^{\pi i\left[\frac{\eta(-\star d)}{4}+\frac{1}{12}\frac{\text{CS}(A^{g})}{2\pi}\right]}.
\end{eqnarray}
where we take $g_{1}:=g$ so that $\star_{g_{1}}:=\star$.

On the other hand, since $\eta(g_{\epsilon})=\eta(-\star_{\epsilon} d)$ by definition, and we know that its limit exists as $\epsilon\rightarrow\infty$ (in fact $\eta_{0}(X,\kappa)=\lim_{\epsilon\rightarrow\infty}\eta(g_{\epsilon})$), we have
\begin{eqnarray}
\lim_{\epsilon\rightarrow\infty}e^{\pi i\left[\frac{\eta(-\star_{\epsilon} d)}{4}+\frac{1}{12}\frac{\text{CS}(A^{g_{\epsilon}})}{2\pi}\right]}=e^{\pi i\left[\frac{\eta_{0}(X,\kappa)}{4}+\left\{\lim_{\epsilon\rightarrow\infty}\frac{1}{12}\frac{\text{CS}(A^{g_{\epsilon}})}{2\pi}\right\}\right]}.
\end{eqnarray}
Thus, we have
\begin{eqnarray}\label{topeq}
e^{\pi i\left[\frac{\eta(-\star d)}{4}+\frac{1}{12}\frac{\text{CS}(A^{g})}{2\pi}\right]}=e^{\pi i\left[\frac{\eta_{0}(X,\kappa)}{4}\right]}e^{\pi i\left\{\lim_{\epsilon\rightarrow\infty}\frac{1}{12}\frac{\text{CS}(A^{g_{\epsilon}})}{2\pi}\right\}}.
\end{eqnarray}
We therefore see that if we can understand the limit $\lim_{\epsilon\rightarrow\infty}\frac{1}{12}\frac{\text{CS}(A^{g_{\epsilon}})}{2\pi}$, we will obtain crucial information for our problem.  The following has been established using a ``Kaluza-Klein'' dimensional reduction technique modeled after the paper \cite{gijp},
\begin{prop}\label{mc2}(\cite{mcl2})
$(X,\phi,\xi,\kappa,g)$ quasi-regular K-contact three-manifold,
\begin{displaymath}
\xymatrix{\xyC{2pc}\xyR{1pc}U(1) \ar@{^{(}->}[r] & X \ar[d]\\
                              & \Sigma}.
\end{displaymath}
Let $g_{\epsilon}:=\epsilon^{-1}\,\kappa\otimes\kappa+\pi^{*}h$.  After choosing a framing for $TX\oplus TX$, corresponding to a choice of vielbeins, then,
\begin{equation}
CS(A^{g_{\epsilon}})=\left(\frac{\epsilon^{-1}}{2}\right)\int_{\Sigma}r\,\omega+\left(\frac{\epsilon^{-2}}{2}\right)\int_{\Sigma}f^{2}\,\omega
\end{equation}
where $r\in C^{\infty}_{orb}(\Sigma)$ is the (orbifold) scalar curvature of $(\Sigma,h)$, $\omega\in\Omega^{2}_{orb}(\Sigma)$ is the (orbifold) Hodge form of $(\Sigma,h)$, and $f:=\star_{h}\omega$.  In particular, the adiabatic limit of $\text{CS}(A^{g_{\epsilon}})$ vanishes:
\begin{equation}
\lim_{\epsilon\rightarrow \infty}\text{CS}(A^{g_{\epsilon}})=0.
\end{equation}
\end{prop}
Proposition \ref{mc2} combined with Eq. \ref{topeq} and Theorem \ref{eta0thm} gives us the following,
\begin{prop}\label{lprop}
$(X,\phi,\xi,\kappa,g)$ closed, quasi-regular K-contact three-manifold.  The counterterm, $C_{T}$, such that $e^{\frac{\pi i}{4}\left[\eta(-\star_{H} D^{1})+C_{T}\right]}$ is a topological invariant that is identically equal to the topological invariant $e^{\pi i\left[\frac{\eta(-\star d)}{4}+\frac{1}{12}\frac{\text{CS}(A^{g})}{2\pi}\right]}$ is $$C_{T}=\frac{1}{512}\int_{X}R^{2}\,\,\kappa\wedge d\kappa.$$
\end{prop}
Given Proposition \ref{lprop} and Theorem \ref{topthm}, we conclude the following as an immediate consequence,
\begin{prop}\label{cprop}
$(X,\phi,\xi,\kappa,g)$ closed, quasi-regular K-contact three-manifold.  Then,
\begin{eqnarray*}
\eta(-\star d)+\frac{1}{3}\frac{\text{CS}(A^{g})}{2\pi}&=&\eta(-\star D)+\frac{1}{512}\int_{X}R^{2}\,\,\kappa\wedge d\kappa\\
                                                       &=&1+\frac{d}{3}+4\sum_{j=1}^{N}s(\alpha_{j},\beta_{j}),
\end{eqnarray*}
where $d=c_1(X)=n+\sum_{j=1}^{N}\frac{\beta_{j}}{\alpha_j}\in\Q$ and
\begin{equation*}
s(\alpha,\beta):=\frac{1}{4\alpha}\sum_{k=1}^{\alpha-1}cot\left(\frac{\pi k}{\alpha}\right)cot\left(\frac{\pi k\beta}{\alpha}\right)\in\Q
\end{equation*}
is the classical Rademacher-Dedekind sum, where $[n; (\alpha_{1},\beta_{1}),\ldots,(\alpha_{N},\beta_{N})]$ (for gcd$(\alpha_{j},\beta_{j})=1$) are the Seifert invariants of $X$.  In particular, we have computed the $U(1)$-Chern-Simons partition function as:
\begin{eqnarray*}
Z_{U(1)}(X,p,k)&=&k^{n_X}e^{\pi i k S_{X,P}(A_{0})}e^{\frac{\pi i}{4}\left(1+\frac{d}{3}+4\sum_{j=1}^{N}s(\alpha_{j},\beta_{j})\right)}\int_{\mathcal{M}_{P}}\,\, (T^{d}_{C})^{1/2},\\
               &=&k^{m_X}e^{\pi i k S_{X,P}(A_{0})}e^{\frac{\pi i}{4}\left(1+\frac{d}{3}+4\sum_{j=1}^{N}s(\alpha_{j},\beta_{j})\right)}\int_{\mathcal{M}_{P}}\,\, (T^{d}_{RS})^{1/2}.
\end{eqnarray*}

\end{prop}

\appendix

\section{Basic construction of $U(1)$-Chern-Simons theory} \label{appen1}

Let $X$ be a closed oriented $3$-manifold.  For any $U(1)$-connection $A\in\mathcal{A}_{P}$, \cite{m} defined an induced $SU(2)$-connection $\hat{A}$ on an associated principal $SU(2)$-bundle $\hat{P}=P\times_{U(1)} SU(2)$.  i.e. $$\hat{A}|_{[p,g]}=Ad_{g^{-1}}(\rho_{*}pr_{1}^{*}A|_{p})+pr_{2}^{*}\vartheta_{g}$$ where $\rho:U(1)\rightarrow SU(2)$ is the diagonal inclusion, $pr_{1}:P\times SU(2)\rightarrow P$ and $pr_{2}:P\times SU(2)\rightarrow SU(2)$.  Since for any $3$-manifold $X$, $\hat{P}$ is trivializable, let $\hat{s}:X\rightarrow \hat{P}$ be a global section.  The definition we use for the Chern-Simons action is as follows:
\begin{define}
The Chern-Simons action functional of a $U(1)$-connection $A\in\mathcal{A}_{P}$ is defined by:
\begin{equation}\label{act}
S_{X,P}(A)=\int_{X}\hat{s}^{*}\alpha(\hat{A})\,\,\,(mod\,\, \Z)
\end{equation}
where $\alpha(\hat{A})\in\Omega^{3}(\hat{P},\R)$ is the Chern-Simons form of the induced $SU(2)$-connection $\hat{A}\in\mathcal{A}_{\hat{P}}$,
\begin{equation}
\alpha(\hat{A})=Tr(\hat{A}\wedge F_{\hat{A}})-\frac{1}{6}Tr(\hat{A}\wedge [\hat{A},\hat{A}])
\end{equation}
\end{define}
We then define the partition function for $U(1)$-Chern-Simons theory as (as in \cite{m}, \cite{mpr}):
\begin{equation}
Z_{U(1)}(X,k)=\sum_{p\in\text{Tors}H^{2}(X;\Z)}Z_{U(1)}(X,p,k)
\end{equation}
where,
\begin{equation}
Z_{U(1)}(X,p,k)=\frac{1}{Vol(\mathcal{G}_{P})}\int_{\mathcal{A}_{P}}\mathcal{D}A e^{\pi i k S_{X,P}(A)}
\end{equation}
and
\begin{equation}
S_{X,P}(A)=\int_{X}\hat{s}^{*}\alpha(\hat{A})
\end{equation}
Then for any principal $U(1)$-bundle $P$ we follow \cite{bw} and define a new action
\begin{equation}
S_{X,P}(A,\Phi):=S_{X,P}(A-\kappa\Phi)
\end{equation}
where we may view $\Phi\in \Omega^{0}(X)$ and,
\begin{eqnarray}
S_{X,P}(A,\Phi)&=&\int_{X}\alpha(\widehat{A-\kappa\Phi})\\
                     &=&\int_{X}\alpha(\hat{A}-\kappa\hat{\Phi})\\
                     &=&S_{X,P}(A)-\int_{X}[2\kappa\wedge Tr (\hat{\Phi}\wedge F_{\hat{A}})-\kappa\wedge d\kappa\,\, Tr (\hat{\Phi}^{2})]
\end{eqnarray}
where the second equality follows from the definition of $\hat{A}$ and $\hat{\Phi}$ (where $\hat{\Phi}|_{[p,g]}=Ad_{g^{-1}}(\rho_{*}pr_{1}^{*}\Phi|_{p})$) on $\hat{P}=P\times_{U(1)} SU(2)$.  The third equality follows from Eq. 3.6 of \cite{bw}.  We then define a new partition function
\begin{equation}
\bar{Z}_{U(1)}(X,p,k):=\frac{1}{Vol(S)}\frac{1}{Vol(\mathcal{G}_{P})}\int_{\mathcal{A}(P)}DA\,D\Phi\,\, e^{\pi i k S_{X,P}(A,\Phi)}
\end{equation}
where $D\Phi$ is defined by the invariant, positive definite quadratic form,
\begin{equation}\label{phiprod}
(\Phi,\Phi)=-\frac{1}{4\pi^{2}}\int_{X}\Phi^{2}\kappa\wedge d\kappa
\end{equation}
As observed in \cite{bw}, our new partition function is identically equal to our original partition function defined for $U(1)$-Chern-Simons theory.  On the one hand, we can fix $\Phi=0$ above using the shift symmetry, $\delta\Phi=\sigma$, which will cancel the pre-factor $Vol(S)$ from the resulting group integral over $S$ and yield exactly our original partition function:
\begin{equation*}
Z_{U(1)}(X,p,k)=\frac{1}{Vol(\mathcal{G}_{P})}\int_{\mathcal{A}_{P}}\mathcal{D}A\,\, e^{\pi i k S_{X,P}(A)}
\end{equation*}
Thus, we obtain the heuristic result,
\begin{equation}\label{equivpart}
\bar{Z}_{U(1)}(X,p,k)=Z_{U(1)}(X,p,k).
\end{equation}
On the other hand, we obtain another description of $\bar{Z}_{U(1)}(X,p,k)$ by integrating $\Phi$ out.  We will briefly review this computation here.  Our starting point is the formula for the shifted partition function
\begin{equation}
\bar{Z}_{U(1)}(X,p,k)=\frac{1}{Vol(S)}\frac{1}{Vol(\mathcal{G}_{P})}\int_{\mathcal{A}(P)}DA\,D\Phi\,\, e^{\pi i k S_{X,P}(A,\Phi)}
\end{equation}
where
\begin{equation}
S_{X,P}(A,\Phi)=S_{X,P}(A)-\int_{X}[2\kappa\wedge Tr (\hat{\Phi}\wedge F_{\hat{A}})-\kappa\wedge d\kappa\,\, Tr (\hat{\Phi}^{2})]
\end{equation}
We formally complete the square with respect to $\hat{\Phi}$ as follows:
\begin{eqnarray*}
\int_{X}[\kappa\wedge d\kappa\,\, Tr (\hat{\Phi}^{2})&-&2\kappa\wedge Tr (\hat{\Phi}\wedge F_{\hat{A}})]\\
&=&\int_{X}\left[Tr (\hat{\Phi}^{2})-\frac{2\kappa\wedge Tr (\hat{\Phi}\wedge F_{\hat{A}})}{\kappa\wedge d\kappa}\right]\kappa\wedge d\kappa\\
                                                                                                      &=&\int_{X}Tr \left(\hat{\Phi}^{2}-\frac{2\kappa\wedge
                                                                                                      F_{\hat{A}}}{\kappa\wedge d\kappa}\hat{\Phi}\right)\kappa\wedge d\kappa\\
                                                                                                      &=&\int_{X}Tr \left(\left[\hat{\Phi}-\frac{\kappa\wedge F_{\hat{A}}}{\kappa\wedge d\kappa}\right]^{2}-\left[\frac{\kappa\wedge F_{\hat{A}}}{\kappa\wedge d\kappa}\right]^{2}\right)\kappa\wedge d\kappa
\end{eqnarray*}
We then only need to compute the Gaussian
\begin{eqnarray*}
\int D\Phi\,\,&\text{exp}&\left[ \pi i k\,\int_{X} \text{Tr} \left(\left[\hat{\Phi}-\frac{\kappa\wedge F_{\hat{A}}}{\kappa\wedge d\kappa}\right]^{2}\right)\kappa\wedge d\kappa\right]\\
            &=&\int D\Phi\,\text{exp}\left[ \pi i k\,\, \int_{X} Tr (\hat{\Phi}^{2})\kappa\wedge d\kappa\right]\\
            &=&\int D\Phi\,\text{exp}\left[ \frac{i k}{4\pi}\,\, \int_{X}\Phi^{2}\kappa\wedge d\kappa\right]\\
            &=&\int D\Phi\,\text{exp}\left[ -\frac{1}{2}\,\, (\Phi, A\Phi) \right]
\end{eqnarray*}
where we take $A=2\pi i k\I$ acting on the space of fields $\Phi$ and the inner product $(\Phi,\Phi)$ is defined as in Eq. \ref{phiprod}.  We then formally get
\begin{eqnarray}
\int D\Phi\,\text{exp}\left[ -\frac{1}{2}\,\, (\Phi, A\Phi) \right]&=&\sqrt{\frac{(2\pi)^{\Delta\mathcal{G}}}{\text{det}A}}\\\label{dimterm}
                                                                   &=&\left(\frac{-i}{k}\right)^{\Delta\mathcal{G}/2}
\end{eqnarray}
where the quantity $\Delta\mathcal{G}$ is formally the dimension of the gauge group $\mathcal{G}$.
Note that we have abused notation slightly throughout by writing $\frac{1}{\kappa\wedge d\kappa}$.  We have done this with the understanding that since $\kappa\wedge d\kappa$ is non-vanishing, then $\kappa\wedge F_{\hat{A}}=\phi\,\kappa\wedge d\kappa$ for some function $\phi\in 2\pi i \Omega^{0}(X)$, and we identify $\frac{\kappa\wedge F_{\hat{A}}}{\kappa\wedge d\kappa}:=\phi$.\\
\\
Our new description of the partition function is now,
\begin{equation}\label{part}
\bar{Z}_{U(1)}(X,p,k)=C\int_{\mathcal{A}_{P}}DA\,\, exp\,\left[\pi i k\left(S_{X,P}(A)-\int_{X} \frac{Tr[(\kappa\wedge F_{\hat{A}})^{2}]}{\kappa\wedge d\kappa}\right)\right]
\end{equation}
where $C=\frac{1}{Vol(S)}\frac{1}{Vol(\mathcal{G}_{P})}\left(\frac{-i}{k}\right)^{\Delta\mathcal{G}/2}$.  We may rewrite this partition function after choosing a flat base point $A_{0}$ in $\mathcal{A}_{P}$ so that $F_{A_{0}}=0$ and identifing $\mathcal{A}(P)=A_{0}+2\pi i \Omega^{1}(X)$.  We then obtain
\begin{equation}\label{Anom1a}
\bar{Z}_{U(1)}(X,p,k)=C_1\int_{\mathcal{A}_{P}}DA\,\, exp\,\left[\frac{i k}{4\pi}\left(\int_{X} A\wedge dA-\int_{X} \frac{(\kappa\wedge dA)^{2}}{\kappa\wedge d\kappa}\right)\right]
\end{equation}
where $$C_1 = \frac{e^{\pi i k S_{X,P}(A_{0})}}{Vol(\mathcal{S})Vol(\mathcal{G}_{P})} \left(\frac{-i}{k}\right)^{\Delta\mathcal{G}/2}.$$
\noindent
We may further simplify Eq. \ref{Anom1a} by reducing $\mathcal{A}_{P}$ to its quotient under the shift symmetry $\bar{\mathcal{A}}_{P}:=\mathcal{A}_{P}/\mathcal{S}$, effectively canceling the factor of $Vol(\mathcal{S})$ out front of the integral. We obtain:
\begin{equation}\label{Anom2a}
\bar{Z}_{U(1)}(X,p,k) =
C_2 \int_{\bar{\mathcal{A}}_{P}}\bar{D}A\,\, exp\,\left[\frac{i k}{4\pi}
\left(\int_{X} A\wedge dA-\int_{X} \frac{(\kappa\wedge dA)^{2}}{\kappa\wedge d\kappa}\right)\right]\\
\end{equation}
where $C_2 = C_1 Vol (S) $.

Note that we are justified in excluding the factor $\left(\frac{-i}{k}\right)^{\Delta\mathcal{G}/2}$ from Eq. \ref{Anom1} since this factor would cancel in the stationary phase approximation in any case.

\section{Framing dependence and the gravitational Chern-Simons term} \label{appen2}

As observed in (\cite{m} ; Eq. 7.17), Eq. \ref{abelchern2} can also be rigorously defined by setting
\begin{equation}
Z_{U(1)}(X,p,k)=\frac{e^{\pi i k S_{X,P}(A_{P})}}{\text{Vol}U(1)}\int_{\mathcal{M}_{P}}e^{\frac{\pi i}{4} \text{sgn}(-\star d)}\frac{[\text{det}'(d^{*}d)]^{1/2}]}{[\text{det}'(-k\star d)]^{1/2}]}\nu
\end{equation}
where $\nu$ is the metric induced on the moduli space of flat connections on $P$, $\mathcal{M}_{P}$. This last expression has rigorous mathematical meaning if the determinants and signatures of the operators are regularized.  The signature of the operator $-\star d$ on $\Omega^{1}(X;\R)$ is regularized via the eta invariant, so that $\text{sgn}(-\star d)=\eta(-\star d)+\frac{1}{3}\frac{\text{CS}(A^{g})}{2\pi}$, where
\begin{equation}
\eta(-\star d)=\lim_{s\rightarrow 0}\sum_{\lambda_{j}\neq 0}sign\lambda_{j}|\lambda_{j}|^{-s}
\end{equation}
and $\lambda_{j}$ are the eigenvalues of $-\star d$, and
\begin{equation}
\text{CS}(A^{g})=\frac{1}{4\pi}\int_{X}Tr(A^{g}\wedge dA^{g}+\frac{2}{3} A^{g}\wedge A^{g}\wedge A^{g})
\end{equation}
is the gravitational Chern-Simons term, with $A^{g}$ the Levi-Civita connection on the spin bundle of $X$.  The determinants are regularized as in Remark 7.6 of \cite{m}.\\
It is straightforward to see that the the term inside of the integral
\begin{equation}
\frac{1}{\text{Vol}U(1)}\frac{[\text{det}'(d^{*}d)]^{1/2}]}{[\text{det}'(-k\star d)]^{1/2}]}
\end{equation}
may be identified with the Reidemeister torsion of the 3-manifold $X$, $T^{d}_{RS}$ (cf. \cite{m} ; Eq. 7.22).  We obtain,
\begin{equation}
Z_{U(1)}(X,p,k)=k^{m_X}e^{\pi i k S_{X,P}(A_{P})}e^{\pi i\left(\frac{\eta(-\star d)}{4}+\frac{1}{12}\frac{\text{CS}(A^{g})}{2\pi}\right)}\int_{\mathcal{M}_{P}}(T^{d}_{RS})^{1/2}
\end{equation}
where $m_X=\frac{1}{2}(\text{dim}H^{1}(X;\R)-\text{dim}H^{0}(X;\R))$.
The Atiyah-Patodi-Singer theorem says that the combination
\begin{equation}
\frac{\eta(-\star d)}{4}+\frac{1}{12}\frac{\text{CS}(A^{g})}{2\pi}
\end{equation}
is a topological invariant depending only on a $2$-framing of $X$.  Recall (\cite{at}) that a~2-framing is choice of a homotopy equivalence class $\pi$ of trivializations of $TX\oplus TX$, twice the tangent bundle of $X$ viewed as a Spin$(6)$ bundle.  The possible $2$-framings correspond to $\Z$.  The identification with $\Z$ is given by the signature defect defined by
\begin{equation}
\delta(X,\pi)=\text{sign}(M)-\frac{1}{6}p_{1}(2TX,\pi)
\end{equation}
where $M$ is a $4$-manifold with boundary $X$ and $p_{1}(2TX,\pi)$ is the relative Pontrjagin number associated to the framing $\pi$ of the bundle $TX\oplus TX$.  The canonical $2$-framing $\pi^{c}$ corresponds to $\delta(X,\pi^{c})=0$.  Either we can choose the canonical framing, and work with this throughout, or we can observe that if the framing of $X$ is twisted by $s$ units, then $CS(A^{g})$ transforms by
\begin{equation}
CS(A^{g})\rightarrow CS(A^{g})+2\pi s
\end{equation}
and so the partition function $Z_{U(1)}(X,k)$ is transformed by
\begin{equation}\label{partform}
Z_{U(1)}(X,k)\rightarrow Z_{U(1)}(X,k)\cdot \text{exp}\left(\frac{2\pi i s}{24}\right)
\end{equation}
Then $Z_{U(1)}(X,k)$ is a topological invariant of framed, oriented $3$-manifolds, with a transformation law under change of framing.  This is tantamount to a topological invariant of oriented $3$-manifolds without a choice of framing.

%\include{append}\label{appen1}
%\include{append2}\label{appen2}

%\begin{thebibliography}{99}

%\end{thebibliography}

\bibliographystyle{amsalpha}
\bibliography{finalthesis}

\end{document}